\documentclass[]{interact}

\usepackage{epstopdf}
\usepackage[caption=false]{subfig}

\usepackage[numbers,sort&compress]{natbib}
\bibpunct[, ]{[}{]}{,}{n}{,}{,}

\usepackage{amsthm}
\usepackage{amsmath}
\usepackage{amsfonts}

\theoremstyle{plain}
\newtheorem{theorem}{Theorem}[section]
\newtheorem{lemma}[theorem]{Lemma}
\newtheorem{corollary}[theorem]{Corollary}
\newtheorem{prop}[theorem]{Proposition}

\theoremstyle{definition}
\newtheorem{definition}[theorem]{Definition}

\theoremstyle{remark}
\newtheorem{remark}{Remark}

\usepackage{color}

\def\R{\mathbb R}
\def\Z{{\mathbb Z}}

\def\al{\alpha}

\begin{document}


\title{Exponential Dichotomy for Noninvertible Linear Difference Equations}

\author{
\name{F. Battelli\textsuperscript{a}\thanks{F. Battelli. Email: battelli@dipmat.univpm.it}, M. Franca\textsuperscript{b}\thanks{M. Franca. Email: matteo.franca4@unibo.it} and K.~J. Palmer\textsuperscript{c}\thanks{K.~J. Palmer. Email: palmer@math.ntu.edu.tw; Corresponding author, partially supported by Marche Polytechnic University}
}
\affil{\textsuperscript{a}DIISM, Marche Polytecnic University, Ancona Italy; \textsuperscript{b}Dipartimento di Scienze Matematiche, Universit\`a di Bologna Italy; \textsuperscript{c}Department of Mathematics, National Taiwan University, Taipei - Taiwan}
}

\maketitle

\begin{abstract}
	In this article we study exponential dichotomies for noninvertible linear difference equations in finite dimensions. After giving the definition, we study the extent to which the projection $P(k)$ in a dichotomy is unique. For equations on
	$\Z$ it is unique but for equations on $\Z_+$ only its range is unique and for $\Z_-$ only its nullspace.
	Here we strengthen Kalkbrenner's results and give a complete characterization of all possible projections.
	Next we study the possibility of extending the dichotomy to a larger interval. We reproduce the results of P\" otzsche but also show exactly	when the original projection remains unchanged. 	
	Next we prove that the roughness theorem, well known for additive perturbations,
	holds for multiplicative perturbations also. The proof uses ideas of Zhou, Lu and Zhang.
	Finally, following Ducrot, Magal and Seydi, we mention that the results by Palmer on finite time conditions on dichotomy for the invertible case can be extended to the noninvertible case.

\end{abstract}




\section{Introduction}

In this article we study exponential dichotomies for linear difference equations in finite dimensions in which the coefficient matrix is neither assumed to be bounded nor invertible. Note that such equations can arise even in the context of invertibility. If we start with an invertible system where the inverse of the coefficient matrix is not uniformly bounded and which has an exponential dichotomy in the usual sense, then an arbitrarily small perturbation of it may not be invertible. However it turns out that the exponential dichotomy property does not disappear. In fact, the perturbed system has an exponential dichotomy in the sense given below. This definition is consistent with that given by Henry \cite{H}. Also it is the same as the regular exponential dichotomy as defined by Kalkbrenner \cite{K}, by P\" otzsche in \cite{PC} and in Ducrot et al. in \cite{DMS1}. Other definitions of exponential dichotomy have been given in, for example, \cite{AK} and \cite{S}, but we feel the definition below is the most appropriate one for the reasons just given.
\medskip

In Section 2, we give two equivalent definitions of exponential dichotomy. Then in some remarks we prove some identities that are important later. It is possible that some of these have not been noticed before.
\medskip

Then in Section 3 we study the extent to which the projection $P(k)$ in a dichotomy is unique. For equations on $\Z_+$, we show in Proposition \ref{prop1} that the range of $P(k)$ is uniquely determined; however, corresponding to any complement $W$ of the range of $P(0)$, there is an invariant projection $Q(k)$, with respect to which the equation has a dichotomy, such that the nullspace of $Q(0)$ is $W$. We prove an analogous result for equations on $\Z_-$ with range replaced by nullspace (note that in the noninvertible case, the results for $\Z_-$ cannot be deduced from those for $\Z_+$ simply by reversing the time; in fact, as we see below, sometimes the results are different). For dichotomies on $\Z$ we deduce that the projection is uniquely determined. The results for $\Z_+$ and $\Z_-$ strengthen Kalkbrenner's results but, of course, we must remember that Kalkbrenner was working in infinite dimensions. As far as we can see, Kalkbrenner only proved that for $\Z_+$ (resp. $\Z_-$), $P(k)$ could be replaced by any invariant projection $Q(k)$ with the same range (resp. nullspace). However he did not show that $Q(0)$ could be any projection with the same range (resp. nullspace) as $P(0)$.
\medskip

In Proposition \ref{propnew} we strengthen these results further. If an equation has a dichotomy on $\Z_+$, then given any $m\ge 0$ and a complement $W$ of the range of $P(m)$ there is an invariant projection $Q(k)$, with respect to which the equation has a dichotomy on $\Z_+$, and such that the range of $Q(m)$ equals that of $P(m)$ and its nullspace is $W$. However,  in contrast to the case where the $A(k)$  are invertible, if we assign the nullspace of $Q(m)$ with $m \ge 1$ the projections $Q(k)$ for $0 \le k \le m-1$ are not in general uniquely defined. There is an analogous result for $\Z_-$ but here there is a difference. If an equation has a dichotomy on $\Z_-$, then given any $m\le 0$ and a complement $W$ of the nullspace of $P(m)$ which {\it contains the nullspace of $\Phi(0,m)$} ($\Phi(m,n)$ is the transition matrix), there is an invariant projection $Q(k)$, with respect to which the equation has a dichotomy on $\Z_-$, and such that the nullspace of $Q(m)$ equals that of $P(m)$ and its range is $W$.
  Again, in contrast to the case where the $A(k)$  are invertible, if we assign the nullspace of $Q(m)$ with $m\le -1$ the projections $Q(k)$ for $m+1\le k\le 0$ are not in general uniquely defined.
\medskip

In Section 4 we study the possibility of extending the dichotomy to a larger interval.
First we show if an equation has a dichotomy on $[m,\infty)$ with projection
$P(k)$ of rank $r$, where $m>0$, then it can be extended to $\Z_+$ if and only if dim $(\Phi(m,0)^{-1}({\cal R}P(m))))=r$. This was proved already by P\" otzsche \cite{PC1} (inspired by Lin \cite{L}) but here we show that in
addition the dichotomy can be extended without changing the
projection when $k\ge m$. The results for $\Z_-$ are rather different.
For $m<0$, suppose \eqref{lineq}
has an exponential dichotomy on $k\le m$ with projection $P(k)$   of rank $r$.
Then \eqref{lineq} has an exponential dichotomy on $k\le 0$ with projection of
rank $r$ if and only if $\Phi(0,m)$ is one to one on ${\cal N}P(m)$; \eqref{lineq} has an exponential dichotomy on $k\le 0$ with projection $P(k)$ for $k\le m$ if and only if $\Phi(0,m)$ is one to one on ${\cal N}P(m)$ and ${\cal N}(\Phi(0,m))\subset {\cal R}P(m)$. Note that the first statement here was proved already by P\" otzsche \cite{PC1} (inspired by Lin \cite{L}) but the second statement appears
	to be new.
\medskip

In Section 5, we prove the roughness theorem. For additive perturbations, that is,
$A(k)$ is perturbed to $A(k)+B(k)$ where $B(k)$ is small, this theorem has already been
proved  by, among others, Henry \cite{H}, Kalkbrenner \cite{K}, P\" otzsche \cite{PC}, Ducrot et al. \cite{DMS1}
  and Zhou, Lu and Zhang \cite{ZLZ}.
The novelty here is that we also prove that the theorem holds for multiplicative perturbations $A(k)[I+B(k)]$   so that roughness of exponential dichotomy holds for both additive and multiplicative perturbations. This is in contrast to the situation with exponential separation. It was shown in \cite{BPD} that, in general, the roughness theorem
  for exponential separation  holds only for multiplicative perturbations.
\medskip

In Section 6, we study finite time conditions on dichotomy.
We could extend what was done in \cite{PK1} for the invertible case to the noninvertible case. However this has already been done in \cite{DMS} in a more general situation. We content ourselves here with stating the main result.
\section{Exponential dichotomy: definition}

We consider the system
\[x(k+1)=A(k)x(k),\quad x(k)\in \R^n,\quad  k\in J',\]
where if $J$ is an infinite interval of integers, we take $J'=(-\infty,{\rm max}\, J-1]$ if $J$ is bounded above and otherwise we take $J'=J$.
  This means that if $x(k)$ satisfies \eqref{lineq} for $k\in J'$, then $x(k)$ is defined for $k\in J$.
\medskip

We denote by $\Phi(k,m)=A(k-1)\cdots A(m)$ the associated transition matrix for $k\ge m$ in $J$. In this section, we define an exponential dichotomy.
\medskip

\begin{definition}\label{EDdef} We say the system
\begin{equation}\label{lineq}
x(k+1)=A(k)x(k)
\end{equation}
has an {\it exponential dichotomy} on an infinite interval $J$ of integers if
there is a projection function $P(k)$ of constant rank such that
\begin{equation}\label{k0}
\Phi(k,m)P(m)=P(k)\Phi(k,m),
\end{equation}
for $k\ge m$ in $J$
and there exist constants $L\ge 1$ and $\alpha\,>0$ such that for $k\ge m$ in $J$
\begin{equation}\label{newk1}
|\Phi(k,m)P(m)|\le Le^{-\alpha(k-m)}.
\end{equation}
Also we assume $A(k):{\cal N}P(k)\to {\cal N}P(k+1)$ is invertible
for $k\in J'$ so that $\Phi(k,m):{\cal N}P(m)\to {\cal N}P(k)$ is invertible for $k\ge m$ in $J$ and we assume that for $k\ge m$ in $J$
\begin{equation}\label{newk1a}
|\Phi(m,k)(I-P(k))|\le Le^{\alpha(k-m)},
\end{equation}
where $\Phi(m,k)$ is the inverse of $\Phi(k,m):{\cal N}P(m)\to {\cal N}P(k)$.
\end{definition}
\medskip
Now we give an equivalent definition.
\medskip

\begin{definition}[\bf Equivalent Definition]\label{EqEDdef}
We say system \eqref{lineq} has an {\it exponential dichotomy} on an infinite interval $J$ of integers if there is a projection function $P(k)$ of constant rank such that
\begin{equation}\label{k0a}
|P(k)|,\quad  |I-P(k)|\le M
\end{equation}
for all $k\in J$, \eqref{k0} holds for $k\ge m$
and there exist constants $K\ge 1$ and $\alpha\,>0$ such that for $k\ge m$ in $J$ and $\xi\in{\cal R}P(m)$
\begin{equation}\label{k1}
|\Phi(k,m)\xi|\le Ke^{-\alpha(k-m)}|\xi|.
\end{equation}
Also we assume $A(k):{\cal N}P(k)\to {\cal N}P(k+1)$ is invertible   for $k\in J'$  so that $\Phi(k,m):{\cal N}P(m)\to {\cal N}P(k)$ is invertible for $k\ge m$ in $J$ and we assume
\begin{equation}\label{k1a}
|\Phi(k,m)\xi|\ge K^{-1}e^{\alpha(k-m)}|\xi|,\quad k\ge m,\; \xi\in {\cal N}P(m).
\end{equation}
\end{definition}

Now we show the equivalence of these definitions.
Denoting the inverse of $\Phi(k,m):{\cal N}P(m)\to {\cal N}P(k)$ by $\Phi(m,k)$, we see that \eqref{k1a} implies that
\begin{equation}\label{k1a3}
|\Phi(m,k)\xi|\le Ke^{-\alpha(k-m)}|\xi|,\quad k\ge m,\; \xi\in {\cal N}P(k).
\end{equation}

 From this and \eqref{k1}, it follows that Definition \ref{EqEDdef} implies Definition \ref{EDdef} with $L=KM$.
On the other hand, \eqref{newk1} and \eqref{newk1a} imply \eqref{k0a} and \eqref{k1}  with $K=M=L$. Moreover, it follows from \eqref{newk1a} that when $\xi\in{\cal N}P(m)$
and $k\ge m$,
\[ |\xi|=|\Phi(m,k)(I-P(k))\Phi(k,m)\xi|\le Le^{-\alpha(k-m)}|\Phi(k,m)\xi|\]
so that \eqref{k1a} holds with $K=L$.
Thus Definition \ref{EDdef} implies Definition \ref{EqEDdef} with $M=K=L$.

\medskip

If $P(k)$ has rank $r$ ($0\le r\le n$), we say that \eqref{lineq} has an exponential dichotomy with rank $r$.
\medskip

\begin{remark}\label{rem1}
i) Property \eqref{k0} is referred to as the {\em invariance} of the family of projections $P(k)$ and is equivalent to
\begin{equation}\label{k1a4}
A(k)P(k)=P(k+1)A(k)
\end{equation}
for any $k,k+1\in J$. Indeed \eqref{k1a4} is equation \eqref{k0} with $k+1$ instead of $k$ and $k$ instead of $m$. Next, assuming \eqref{k1a4} we get by induction:
\[\begin{array}{l}
\Phi(k+1,m)P(m) = A(k)\Phi(k,m)P(m)=A(k)P(k)\Phi(k,m)\\
=P(k+1)A(k)\Phi(k,m)=P(k+1)\Phi(k+1,m).
\end{array} \]

ii) Next, a family of projections $P(k)$ is invariant if and only if $A(k):{\cal N}P(k)\to{\cal N}P(k+1)$ and $A(k):{\cal R}P(k)\to{\cal R}P(k+1)$. Indeed if $P(k)$ is invariant, and $x\in{\cal R}P(k)$ we get $A(k)x = A(k)P(k)x=P(k+1)A(k)x\in{\cal R}P(k+1)$; next, if $x\in{\cal N}P(k)$ we get $A(k)x = A(k)(I-P(k))x=(I-P(k+1))A(k)x\in {\cal N}P(k+1)$.  Conversely, if $A(k):{\cal N}P(k)\to{\cal N}P(k+1)$ and $A(k):{\cal R}P(k)\to{\cal R}P(k+1)$ we have, for any $x\in\R^n$:
\[
P(k+1)A(k)P(k)x = A(k)P(k)x
\]
(because $A(k):{\cal R}P(k)\to{\cal R}P(k+1)$) and
\[
P(k+1)A(k)(I-P(k))x  = 0
\]
(because $A(k):{\cal N}P(k)\to{\cal N}P(k+1)$). Adding the two equalities we get
\[
P(k+1)A(k)x = A(k)P(k)x
\]
for any $x\in\R^n$.
\medskip

iii) Note if we assume that for a given $m$, the range of $P(m)$ is the subspace of initial values of solutions bounded for $k\ge m$, then it follows that $\Phi(k,m):{\cal N}P(m)\to{\cal N}P(k)$ is invertible. For suppose there exists $\xi\in {\cal N}P(m)$ such
that $\Phi(k,m)\xi=0$. Then $\Phi(\ell,m)\xi=0$ for all $\ell\ge k$ and so $\Phi(k,m)\xi$ is bounded in $k\ge m$. It follows that
$\xi\in {\cal R}P(m)$ and so $\xi=0$. Thus $\Phi(k,m):{\cal N}P(m)\to{\cal N}P(k)$ is one to one and hence invertible.
\medskip

iv) The invariance of $P(k)$ and invertibility of $A(k):{\cal N}P(k) \to {\cal N}P(k+1)$ imply the properties
\begin{equation}\label{invPk1}
\Phi(k,m)({\cal R}P(m))\subset {\cal R}P(k),\quad
\Phi(k,m)({\cal N}P(m))={\cal N}P(k),\quad k\ge m.
\end{equation}
However, for any $k\ge m$, the stronger property
\begin{equation}\label{invPk2}
\Phi(k,m)^{-1}({\cal R}P(k)) := \{ x\in\R^n \mid \Phi(k,m)x \in {\cal R}P(k)\}={\cal R}P(m)
\end{equation}
is also true. Indeed it follows from \eqref{invPk1} that ${\cal R}P(m)\subset \Phi(k,m)^{-1}({\cal R}P(k))$. To prove $\Phi(k,m)^{-1}({\cal R}P(k))\subset {\cal R}P(m)$, suppose $x\in\R^n$ is such that $\Phi(k,m)x\in{\cal R}P(k)$. We want to prove that $x\in{\cal R}P(m)$. Write $x=u+v$ with $u\in{\cal R}P(m)$ and $v\in{\cal N}P(m)$. Then $\Phi(k,m)(x-u)=\Phi(k,m)v\in{\cal N}P(k)$. But the left hand side belongs to ${\cal R}P(k)$ because $\Phi(k,m)x\in{\cal R}P(k)$ (by the assumption) and $\Phi(k,m)u\in{\cal R}P(k)$ (by invariance). So $\Phi(k,m)v=0$ and hence $v=0$.
Thus $x\in{\cal R}P(m)$.
\medskip

 Also for $k\ge m$ we have the property
	\begin{equation}\label{invPk3} \Phi(k,m)^{-1}({\cal N}P(k))={\cal N}P(m)\oplus{\cal N}(\Phi(k,m)).\end{equation}
To see this, first note that ${\cal N}P(m)\cap{\cal N}(\Phi(k,m))=\{0\}$ follows from the fact that $\Phi(k,m)$ is one to one on ${\cal N}P(m)$. So we need only show that
\[
	{\cal N}P(m)+{\cal N}(\Phi(k,m))=\Phi(k,m)^{-1}({\cal N}P(k)).
\]
First let $u+v\in {\cal N}P(m)+{\cal N}(\Phi(k,m))$. Then $\Phi(k,m)u\in {\cal N}P(k)$ and $\Phi(k,m)v=0$.	This implies that $\Phi(k,m)(u+v)\in {\cal N}P(k)$ and hence that $u+v$ is in $\Phi(k,m)^{-1}({\cal N}P(k))$. Conversely, suppose $x\in \Phi(k,m)^{-1}({\cal N}P(k))$. Then $\Phi(k,m)x$ is in ${\cal N}P(k)$. There exists a unique $u\in {\cal N}P(m)$ such that $\Phi(k,m)u=\Phi(k,m)x$. Then $x-u\in {\cal N}(\Phi(k,m))$. So $x\in {\cal N}P(m)+{\cal N}(\Phi(k,m))$.
\medskip

v) We also note that if $k\ge m$
\begin{equation}\label{invPk4}
	{\cal N}(\Phi(k,m))\subset {\cal R}P(m).
\end{equation}
For suppose $\Phi(k,m)x=0$. Then $x\in \Phi(k,m)^{-1}({\cal R}P(k))$ and hence $x\in {\cal R}P(m)$.
\end{remark}

\section{Changing the projection}

Suppose \eqref{lineq} has an exponential dichotomy with projection $P(k)$. When the dichotomy is on $\Z_+$, ${\cal R}P(0)$ is called the {\it stable subspace}; when the dichotomy is on $\Z_-$, ${\cal N}P(0)$ is called the {\it unstable subspace}. In this section we show that the stable and unstable subspaces are uniquely determined (Lemma \ref{lem0}) but that the projections are not necessarily so (Proposition \ref{prop1}). Then if a system has a dichotomy on both half-axes, we describe what must be added to ensure dichotomy on the whole axis (Corollary \ref{cor1}).
\medskip

First we state the well-known lemma,  which characterizes the stable and unstable subspaces. We include the proof for the sake of
	completeness.

\begin{lemma}\label{lem0} Suppose system \eqref{lineq} has an exponential dichotomy on $J$ with projection $P(k)$. If $J$ contains $[m,\infty)$, then
\[{\cal R}(P(m))=\left\{\xi:\sup_{k\ge m}|\Phi(k,m)\xi|<\infty\right\}.\]
If $J$ contains $(-\infty,m]$, then ${\cal N}(P(m))$ is the set of $\xi$ for which there exists a solution $y(k)$ of \eqref{lineq} bounded in $k\le m$ such that $y(m)=\xi$. Moreover such a solution is unique.
\end{lemma}

\begin{proof} Consider first the case $J$ contains $[m,\infty)$. So, according to Definition \ref{EqEDdef}, we are assuming there is an invariant projection $P(k)$  and positive constants $M$, $K$ and $\alpha$ such that \eqref{k0}, \eqref{k0a}, \eqref{k1} and \eqref{k1a} hold   in $J$. We show that the range of $P(m)$ is
\[
V=\left\{\xi: \sup_{k\ge m}|\Phi(k,m)\xi|<\infty\right\}.
\]
It is clear from \eqref{k1} that the range of $P(m)$ is a subspace of $V$. Suppose $\xi\in V$. Then if $(I-P(m))\xi\neq 0$, by \eqref{k1} and \eqref{k1a},
\[\begin{array}{rl}
|\Phi(k,m)\xi|
&\ge |\Phi(k,m)(I-P(m))\xi|-|\Phi(k,m)P(m)\xi|\\ \\
& \ge K^{-1}e^{\alpha(k-m)}|(I-P(m))\xi|-Ke^{-\alpha(k-m)}|P(m)\xi| \\ \\
&\to \infty \quad{\rm as}\quad k\to\infty.
\end{array}\]
Hence $(I-P(m))\xi=0$ and so $\xi\in{\cal R}P(m)$. Thus the range of $P(m)$ is $V$, as claimed.
\medskip

Next consider the case $J$ contains $(-\infty,m]$. So we are assuming there is an invariant projection $P(k)$  and positive constants $M$, $K$ and $\alpha$ such that \eqref{k0}, \eqref{k0a}, \eqref{k1} and \eqref{k1a} hold.
In this case we define $V\subset\R^n$ as the subspace of those $\xi\in\R^n$ such that the equation
\[x(k+1)=A(k)x(k), \quad x(m)=\xi\]
has a solution bounded on $k\le m$.
Let $\xi\in {\cal N}P(m)$. Then with $\Phi(k,m)$ as the inverse of $\Phi(m,k):{\cal N}P(k)\to {\cal N}P(m)$, set
$y(k):=\Phi(k,m)\xi\in {\cal N}P(k)$. Let $A(j)^{-1}$ be the inverse of $A(j):{\cal N}P(j)\to {\cal N}P(j+1)$, $j=k,\ldots,m-1$. Then for $k<m$ we have:
\[\begin{array}{l}
y(k+1) = \Phi(k+1,m)\xi = [A(m-1)\cdot\ldots\cdot A(k+1)]^{-1}\xi =\\
A(k)[A(m-1)\cdot\ldots\cdot A(k)]^{-1}\xi = A(k)y(k).
\end{array}\]
Next it follows from \eqref{k1a} that
\[
|\xi|=|\Phi(m,k)\Phi(k,m)\xi|\ge K^{-1}e^{\alpha(m-k)}|\Phi(k,m)\xi|,\quad m\ge k.
\]
So $y(k)$ is a solution of \eqref{lineq} such that $|y(k)|\le Ke^{-\alpha(m-k)}|\xi| \le K|\xi|$ and $y(m)=\xi$. So ${\cal N}P(m)\subset V$.
\medskip

Conversely, suppose $\xi\in V$ and let
$y(k)$ be a bounded solution for $k\le m$ with $y(m)=\xi$. Then
\[ |P(m)y(m)|=|P(m)\Phi(m,k)y(k)|=|\Phi(m,k)P(k)y(k)|
\le KMe^{-\alpha(m-k)}|y(k)|.\]
Since the right side $\to 0$ as $k\to -\infty$, it
follows that $P(m)y(m)=0$ and so $\xi=y(m)\in {\cal N}P(m)$.
Hence $V\subset {\cal N}P(m)$. Thus
the nullspace of $P(m)$ is $V$, as claimed.
\medskip

We must also show that for all $\xi\in V={\cal N}P(m)$, there is a unique solution $u(k)$ defined and bounded for $k\le m$ such that $u(m)=\xi$. We emphasize that the uniqueness problem arises since $\Phi(m,k)$ might not be invertible in $\mathbb{R}^n$ for $m>k$. We already showed above that such a solution exists and is given by $y(k)=\Phi(k,m)\xi$, with $\Phi(k,m)$ as the inverse of $\Phi(m,k):{\cal N}P(k)\to {\cal N}P(m)$. To show it is unique, let  $u(k)$ be a solution defined and bounded for $k\le m$  with $u(m)=\xi\in{\cal N}P(m)$. Then $u(k)=\Phi(k,p)u(p)$ for $p\le k\le m$. So
\[
|P(k)u(k)|=|\Phi(k,p)P(p)u(p)|\le Ke^{-\alpha(k-p)}|P(p)u(p)|\to 0\;{\rm as}\; p\to -\infty.
\]
It follows that $P(k)u(k)=0$ for $k\le m$ so that
$u(k)\in{\cal N}P(k)$. Now $u(m)=\Phi(m,k)u(k)$. So, with $\Phi(k,m)$ as the inverse of $\Phi(m,k):{\cal N}P(k)\to {\cal N}P(m)$, $u(k)=\Phi(k,m)u(m)=\Phi(k,m)\xi=y(k)$. Thus we have the uniqueness.
\medskip

\end{proof}

  In the next proposition, we consider a system with an
	exponential dichotomy on $\Z_+$, $\Z_-$ or $\Z$ and determine the
	extent to which the corresponding projections are unique.

\begin{prop} \label{prop1}   Let \eqref{lineq} have an exponential
		dichotomy on an interval $J$ with projection $P(k)$. When $J=\Z_+$, the stable subspace,   that is, the range of $P(0)$, is uniquely defined. The nullspace of $P(0)$  can be any complement, that is, if $W$ is any complement to ${\cal R}(P(0))$, there is a unique invariant projection $Q(k)$ such that its nullspace at $k=0$ is $W$ and \eqref{lineq} has an exponential dichotomy with $Q(k)$ as projection;
	      moreover, such a $Q(k)$ has the same range as $P(k)$ for all $k\ge 0$.
\medskip

For $J=\Z_-$, the unstable subspace   or nullspace of $P(0)$  is uniquely defined. The
 range of $P(0)$  can be any complement, that is, if $W$ is any complement to ${\cal N}(P(0))$, there is a unique invariant projection $Q(k)$ such that its range at $k=0$ is $W$ and \eqref{lineq} has an exponential dichotomy with $Q(k)$ as projection;
  moreover, such a $Q(k)$ has the same nullspace as $P(k)$ for all $k\le 0$.
\medskip

For $\Z$, the projection is uniquely defined.
\end{prop}

\begin{proof} Consider first $J=\Z_+$. So, according to Definition \ref{EqEDdef}, we are assuming there is an invariant projection $P(k)$  and positive constants $M$, $K$ and $\alpha$ such that \eqref{k0}, \eqref{k0a}, \eqref{k1} and \eqref{k1a} hold
in $\Z_+$. That the stable subspace is uniquely defined follows from Lemma \ref{lem0}.
\medskip

Now we show the other subspace can be any complement $W$. We define $W(k)=\Phi(k,0)(W)$. It is obvious that
$A(k):W(k)\to W(k+1)$. First we show by induction on $k$ that ${\cal R}P(k)\cap W(k)=\{0\}$ for $k\ge 0$. This is true for $k=0$. Assume it is true for some $k\ge 0$ and let $x\in {\cal R}P(k+1)\cap W(k+1)$. Then there exists $y\in W(k)$ such that $x=A(k)y$. Then, by invariance, $A(k)P(k)y=P(k+1)A(k)y=P(k+1)x$, so that
\[
A(k)(I-P(k))y=x-P(k+1)x=0,
\]
from which it follows that $(I-P(k))y=0$, because $(I-P(k))y\in{\cal N}P(k)$ and $A(k)$ is invertible on ${\cal N}P(k)$. So $y\in{\cal R}P(k)\cap W(k)$. Hence $y=0$ and therefore $x=0$ also. Thus ${\cal R}P(k)\cap W(k)=\{0\}$ for $k\ge 0$. Next we show that $A(k)$ is one to one on $W(k)$. For suppose $x\in W(k)$ satisfies $A(k)x=0$.
Then from \eqref{invPk4}, $x \in {\cal N}\Phi(k+1,k) \subset {\cal R} P(k)$ so that  $x=0$. Therefore $A(k)$ maps $W(k)$ bijectively on to $W(k+1)$ and so all the $W(k)$ have the same dimension as $W$ and therefore ${\cal R}P(k)\oplus W(k)=\R^n$ for $k\ge 0$.
\medskip

Then we define $Q(k)$ as the projection having the same range as $P(k)$ and nullspace $W(k)$. Since ${\cal R}P(k)$ and $W(k)$ are invariant under $A(k)$ it follows from Remark \ref{rem1}-ii) that $Q(k)$ is an invariant family of projections. Moreover we have seen that $A(k):W(k)={\cal N}Q(k)\to W(k+1)={\cal N}Q(k+1)$ is invertible.

\medskip

Now we prove the boundedness of the projection and the dichotomy inequalities with $Q(k)$ as the projection.
Observe that for $k\ge m\ge 0$ and $\xi\in{\cal R}Q(m)$, using
\eqref{k0a} and \eqref{k1},
\begin{equation}\label{qst} |\Phi(k,m)\xi|= |\Phi(k,m)P(m)\xi|\le KMe^{-\alpha(k-m)}|\xi|.
\end{equation}
Next for all $\xi$ and $k\ge 0$,
\begin{equation}\label{qpdiff}\begin{array}{rl}
|[Q(k)-P(k)]\xi|
&=|Q(k)(I-P(k))\xi|\quad {\rm using}\quad Q(k)P(k)=P(k)\\ \\
&=|Q(k)\Phi(k,0)\Phi(0,k)(I-P(k))\xi|,\\ \\
&\qquad {\rm where\;} \Phi(0,k)\;{\rm is\; the\; inverse\; of\;} \Phi(k,0):{\cal N}P(0)\to{\cal N}P(k)\\ \\
&=|\Phi(k,0)Q(0)\Phi(0,k)(I-P(k))\xi|\\ \\
&\le Ke^{-\alpha k}|Q(0)|Ke^{-\alpha k}|(I-P(k))\xi|\quad {\rm using\;} \eqref{k1}, \eqref{k1a3}\\ \\
&\le K^2M|Q(0)|e^{-2\alpha k}|\xi|.
\end{array}\end{equation}
In particular, this implies that for $k\ge 0$,
\begin{equation}\label{qbd}|Q(k)|\le M_1=M+K^2M|Q(0)|.\end{equation}
Next note, using $Q(k)P(k)=P(k)$ again, that
\[ |(I-Q(k))x|=|(I-Q(k))(I-P(k))x| \le |I-Q(k)||(I-P(k))x|\le (1+M_1)|(I-P(k))x|\]
so that for $k\ge 0$
\begin{equation}\label{pqeq} |(I-P(k))x|\ge (1+M_1)^{-1}|(I-Q(k))x|.\end{equation}
Finally if $k\ge m$ and $\xi\in{\cal N}Q(m)$,
\[\begin{array}{rl}
|\Phi(k,m)\xi|
&\ge \displaystyle\frac{|(I-P(k))\Phi(k,m)\xi|}{|I-P(k)|}\\ \\
&\ge \displaystyle\frac{|\Phi(k,m)(I-P(m))\xi|}{M}\\ \\
&\ge \displaystyle\frac{K^{-1}e^{\alpha(k-m)}\,|(I-P(m))\xi|}{M}\quad{\rm using\;}\eqref{k1a3}\\ \\
&\ge \displaystyle\frac{1}{KM(1+M_1)}e^{\alpha(k-m)} \,|(I-Q(m))\xi|\quad{\rm using\;}\eqref{pqeq}\\ \\
&=\displaystyle\frac{1}{KM(1+M_1)}e^{\alpha(k-m)} \,|\xi|.
\end{array}\]
The boundedness of the projection was proved in \eqref{qbd} and the dichotomy inequalities   as in the equivalent definition  follow from the last inequality and \eqref{qst}. Hence the equation has an exponential dichotomy on $\Z_+$ with $Q(k)$ as the projection.   Note that the range of $Q(k)$ for any $k\ge 0$ must equal the range of $P(k)$, by Lemma \ref{lem0}. Also the nullspace must be ${\cal N}Q(k)=\Phi(k,0){\cal N}Q(0)=\Phi(k,0)(W)$. So $Q(k)$ is unique.
This completes the proof for $J=\Z_+$.
\medskip	

Now we turn to $\Z_-$. So we are assuming there is an invariant projection $P(k)$  and positive constants $M$, $K$ and $\alpha$ such that \eqref{k0}, \eqref{k0a}, \eqref{k1} and \eqref{k1a} hold. That the unstable subspace is uniquely defined follows from Lemma \ref{lem0}.
\medskip

Now we show that the range of $Q(0)$ can be taken as an arbitrary complement of the nullspace of $P(0)$.
Let $W$ be any complement of ${\cal N}P(0)$. We define
\[
W(k)=\{\xi: \Phi(0,k)\xi\in W\}=\Phi(0,k)^{-1}(W),\quad k\le 0.
\]
It is easy to see that $W(k)$ is invariant. Indeed let $x\in W(k)$   for $k\le -1$.
 Then $\Phi(0,k)x\in W$ that is $\Phi(0,k+1)A(k)x\in W$. But then $A(k)x\in W(k+1)$.

Next we prove that $W(k)\cap {\cal N}P(k)=\{0\}$
  for $k\le 0$. For let $x\in W(k)\cap {\cal N}P(k)$. Then $\Phi(0,k)x\in W$ and $\Phi(0,k)x\in {\cal N}P(0)$. So $\Phi(0,k)x=0$ since $W\cap {\cal N}P(0)=\{0\}$. Hence $x=0$ since $\Phi(0,k):{\cal N}P(k)\to{\cal N}P(0)$ is invertible. Thus $W(k)\cap {\cal N}P(k)=\{0\}$.

Next we show that $W(k)+{\cal N}P(k)=\R^n$. For let $x\in\R^n$. Then
\[
\Phi(0,k)x=w+u,\quad {\rm where}\quad w\in W,\; u\in {\cal N}P(0).
\]
We can write $u=\Phi(0,k)v$, where $v\in {\cal N}P(k)$. Then $w=\Phi(0,k)(x-v)$. This means that $x-v\in W(k)$. So
\[
x=x-v+v,\quad {\rm where}\quad x-v\in W(k),\; v\in {\cal N}P(k).
\]
It follows then that $W(k)+{\cal N}P(k)=\R^n$ and hence that
\[
W(k)\oplus{\cal N}P(k)=\R^n.
\]
Now we take $Q(k)$ as the projection with range $W(k)$ and nullspace ${\cal N}P(k)$. It follows from Remark \ref{rem1}-ii)
that $Q(k)$ is an invariant projection such that $A(k)$ maps ${\cal N}Q(k)$ onto ${\cal N}Q(k+1)$.
\medskip

Now we prove the boundedness of the projection and the dichotomy inequalities with $Q(k)$ as projection. First note that for $0\ge k\ge m$, $\Phi(k,m)\xi=\Phi(k,0)\Phi(0,m)\xi$ for all $\xi\in{\cal N}P(m)$ and that for all $\xi$, $[P(m)-Q(m)]\xi\in{\cal N}P(m)$. It follows that for $0\ge k\ge m$ and all $\xi$,
\begin{equation} \label{Q4} \begin{array}{rl}
|\Phi(k,m)[P(m)-Q(m)]\xi|
&=|\Phi(k,0)\Phi(0,m)[P(m)-Q(m)]\xi|\\ \\
&=|\Phi(k,0)[P(0)-Q(0)]\Phi(0,m)\xi|\quad{\rm by \; invariance}\\ \\
&=|\Phi(k,0)[P(0)-Q(0)]P(0)\Phi(0,m)\xi|\\ \\
& \qquad{\rm using\;}
Q(0)P(0)=Q(0)\\ \\
&=|\Phi(k,0)[P(0)-Q(0)]\Phi(0,m)P(m)\xi|\\ \\
&\le Ke^{\alpha k}|P(0)-Q(0)|Ke^{\alpha m}|P(m)\xi|\\ \\
&\qquad {\rm using} \; (\ref{k1}), (\ref{k1a3}){\;\rm and\;}{\cal R}(P(0)-Q(0))
\in{\cal N}P(0)\\\\
& \le K^2Me^{\alpha(k+m)}|P(0)-Q(0)|\,|\xi| \quad {\rm using\;} (\ref{k0a}).
\end{array}\end{equation}
 In particular, with $k=m$, we see that we get for $k\le 0$
\begin{equation}\label{qpdiff1} |P(k)-Q(k)|\le K^2Me^{2\alpha k}|P(0)-Q(0)|
\end{equation}
so that
\begin{equation}\label{Q1} |Q(k)|,\,|I-Q(k)| \le M+K^2M
|P(0)-Q(0)|.\end{equation}
Next for $0\ge k\ge m$ and $\xi\in{\cal N}Q(k)={\cal N}P(k)$, it follows from
\eqref{k1a} that
\begin{equation}\label{Q2}|\Phi(k,m)\xi| \ge K^{-1}e^{\alpha(k-m)}|\xi|.\end{equation}
Finally for $0\ge k\ge m$ and $\xi\in{\cal R}Q(m)$, using \eqref{k0}, \eqref{k0a}
and \eqref{Q4},
\begin{equation}\label{Q3}\begin{array}{rl}|\Phi(k,m)\xi|
&\le |\Phi(k,m)P(m)\xi|+|\Phi(k,m)(P(m)-Q(m))\xi|\\ \\
&\le KMe^{-\alpha(k-m)}|\xi| + K^2Me^{\alpha(m+k)}|P(0)-Q(0)|\,|\xi|\\ \\
&\le KM[1+K
|P(0)-Q(0)|]e^{-\alpha(k-m)}|\xi|.
\end{array}\end{equation}
\eqref{Q1}, \eqref{Q2} and \eqref{Q3} show we have an exponential dichotomy with $Q(k)$ as the
projection.   Note that the nullspace of $Q(k)$ for any $k\ge 0$ must equal
the nullspace of $P(k)$, by Lemma \ref{lem0}. Also, using \eqref{invPk2}, the range
of $Q(k)$ must be $\Phi(0,k)^{-1}{\cal R}Q(0)=\Phi(0,k)^{-1}(W)=W(k)$. So $Q(k)$ is unique.
This completes the proof for $J=\Z_-$.
\medskip

The uniqueness of the projection in the case of $\Z$ follows from the first two parts. So the proof of Proposition \ref{prop1} is concluded.
\end{proof}

\begin{remark}\label{newrem} It follows from \eqref{qpdiff} (resp.
		\eqref{qpdiff1}) that $|Q(k)-P(k)|\to 0$ exponentially fast
		as $k\to\infty$ (resp. $k\to -\infty$).		
\end{remark}

From Proposition \ref{prop1}, we deduce a corollary which has been proved in a different way in \cite{BDV} (see Theorems 2.3 and 2.9). Proposition 4.3 in \cite{PC1} considers the more general situation when the ranks 
on $\Z_+$ and $\Z_-$ may be different.
Proposition 2.10 in \cite{BDV} also considers a more general situation.  Our aim here was simply to consider those systems which have dichotomies on both half-axes and ask what has to be added to ensure dichotomy on the whole axis. 

\begin{corollary} \label{cor1} System (\ref{lineq}) has an exponential dichotomy on $\Z$\ if and only if it has an exponential dichotomy on $\Z_+$ and $\Z_-$, the respective ranks are the same and the stable subspace on $\Z_+$ and the unstable subspace on $\Z_-$ intersect in $\{0\}$ at $k=0$.
\end{corollary}

\begin{proof} Clearly the conditions are necessary.
 \medskip

 For the sufficiency, suppose (\ref{lineq}) has an exponential dichotomy
 on $\Z_+$ and $\Z_-$, the  respective ranks are the same and the stable subspace on $\Z_+$ and the unstable subspace on $\Z_-$ intersect in $\{0\}$ at $k=0$. According to Proposition \ref{prop1}, at $k=0$, we can take the unstable subspace on $\Z_{+}$ to be the unstable subspace on $\Z_{-}$ and the stable subspace on $\Z_{-}$ to be  the stable subspace on $\Z_{+}$.
 This means we have the same invariant projection $P(0)$ on both $\Z_+$ and $\Z_-$
 and we will have bounds as in \eqref{k0a} on $\Z$.   Also $A(k)$ will map ${\cal N}P(k)$ bijectively on to
 		${\cal N}P(k+1)$ for $k\in\Z$.
 Next there exist positive constants $K$ and $\alpha$ such that for $\xi\in{\cal R}P(m)$
 and $\eta\in{\cal N}P(m)$,
 \[ |\Phi(k,m)\xi|\le Ke^{-\alpha(k-m)}|\xi|, \quad
 |\Phi(k,m)\eta| \ge K^{-1}e^{\alpha(k-m)}|\eta|\]
 for $k\ge m\ge 0$ and $0\ge k\ge m$. Then if $k\ge 0\ge m$ and $\xi\in{\cal R}P(m)$,
 \[ |\Phi(k,m)\xi|=|\Phi(k,0)\Phi(0,m)\xi|\le Ke^{-\alpha k}Ke^{-\alpha(- m)}|\xi|
  =K^2e^{-\alpha(k-m)}|\xi|
 \]
 and if $\eta\in{\cal N}P(m)$,
 \[|\Phi(k,m)\eta|=|\Phi(k,0)\Phi(0,m)\eta|
 \ge K^{-1}e^{\alpha k}K^{-1}e^{-\alpha m}|\eta|=K^2e^{\alpha(k-m)}|\eta|.\]
  It follows that (\ref{lineq}) has an exponential dichotomy on $\Z$.

 \end{proof}

  Now we prove some further results about the projections for dichotomies
	on a half-axis.

\begin{prop}\label{propnew} {\rm (i)} Assume \eqref{lineq} has an exponential dichotomy on $\Z_+$ with projection $P(k)$. Fix $m \in \Z_+$
	and let $W$ be any complement of ${\cal R}P(m)$. Then there is an invariant projection
	$Q(k)$ such that ${\cal R}Q(m)={\cal R}P(m)$, ${\cal N}Q(m)=W$ and \eqref{lineq} admits an exponential dichotomy on $\Z_+$  with projection $Q(k)$.
	\medskip

{\rm (ii)} Assume \eqref{lineq} has an exponential dichotomy on $\Z_-$ with projection $P(k)$. Fix $m \in \Z_-$ and let $W$ be any complement of ${\cal N}P(m)$ such that ${\cal N}(\Phi(0,m))\subset W$. Then there is an invariant projection $Q(k)$ such that ${\cal N}Q(m)={\cal N}P(m)$, ${\cal R}Q(m)=W$ and \eqref{lineq} admits an exponential dichotomy on $\Z_-$  with projection $Q(k)$.

\end{prop}

\begin{proof}
(i) The $m=0$ case follows from Proposition \ref{prop1}. So let $m>0$. Let dim$({\cal R} P(k))=r$ so that dim$(W)=\textrm{dim}(
{\cal N} P(k))=n-r$ for any $k \ge 0$. Let $V$ be a linear space such that
$$V \oplus {\cal N} (\Phi(m,0))=  \Phi(m,0)^{-1}(W).$$
Then $\Phi(m,0)$ restricted to $V$ is bijective from $V$ to $W$. In fact it is one to one since $V \cap {\cal N} (\Phi(m,0))= \{0 \}$, and it is onto since, if $w \in W$, there are $v \in V$ and $y \in {\cal N} (\Phi(m,0))$ such that $w=\Phi(m,0)(v+y)=\Phi(m,0)v$.
So dim$(V)= \textrm{dim}(W)=n-r$. Furthermore $V \cap {\cal R}P(0)= \{ 0 \}$: in fact if $v \in V \cap {\cal R}P(0)$ then
$\Phi(m,0)v=w \in W$ by construction and $w \in {\cal R}P(m)$ by invariance, see \eqref{invPk1}, thus $w \in W \cap {\cal R}P(m)
=\{ 0\}$; whence $v \in {\cal N}(\Phi(m,0))$, but since $v \in V$ we get $v=0$. Therefore $V \oplus {\cal R}P(0)= \R^n$.
From Proposition \ref{prop1} we see that there is a unique projection $Q(k)$ such that ${\cal R}Q(0)={\cal R}P(0)$ and ${\cal N}Q(0)=V$, and \eqref{lineq} admits an exponential dichotomy with projection $Q(k)$. Now from   Lemma \ref{lem0}  we get ${\cal R}Q(m)={\cal R}P(m)$ and from \eqref{invPk1} we get ${\cal N}Q(m)=\Phi(m,0)V=W$ and (i) is proved.
   \medskip

   (ii)   Let $X= \Phi(0,m)W$. Note that
   \[X\cap{\cal N}P(0)=\{0\}, \]
   since if $x=\Phi(0,m)w$, $w\in W$, is in ${\cal N}P(0)=\Phi(0,m)({\cal N}P(m))$,
   then $x=\Phi(0,m)u$, where $u\in {\cal N}P(m)$. Then $u-w\in {\cal N}(\Phi(0,m))\subset W$ so that $u\in W\cap {\cal N}P(m)$. Hence $u=0$ and so $x=0$ also.
   \medskip

   Then we can let $V$ be a linear space such that
   $$V\oplus {\cal N}P(0)= \R^n$$
   and $X\subset V$.  From Proposition
   \ref{prop1} there is a unique projection $Q(k)$ such that
   $${\cal N}Q(k)={\cal N}P(k) \, , \qquad {\cal R}Q(0)=V$$
   for any $k \in \Z_-$, and \eqref{lineq} admits an exponential dichotomy with projection $Q(k)$.
   In particular ${\cal N}Q(m)={\cal N}P(m)$; further    from \eqref{invPk2} we see that
   $${\cal R}Q(m)=\Phi(0,m)^{-1}(V)\supset W.$$
   However since ${\cal N}Q(m) \oplus {\cal R}Q(m) =\R^n= {\cal N}Q(m) \oplus W$ we see that
   $W={\cal R}Q(m)$ and (ii) is proved.
   \end{proof}
\begin{remark}
We recall that if we are in the hypotheses of Proposition \ref{propnew} (ii), but $A(k)$ is invertible for any $k \le 0$, then we
do not have any restriction on the choice of $W$, i.e. we can simply take $W$ to be any complement of $\mathcal{N}P(m)$.
However, when the invertibility of $A(k)$ is not ensured,  by \eqref{invPk4}, we need to require $\mathcal{N}(\Phi(0,m)) \subset W$ (clearly such a condition is trivially satisfied if $m=0$ or if $\Phi(0,m)$ is invertible).
\end{remark}

  Suppose \eqref{lineq} has an exponential dichotomy with projections $P(k)$ and $Q(k)$
such that $P(m)=Q(m)$ for some $m$. We recall that if $A(k)$ is invertible for any $k$ it follows that $P(k)=Q(k)$ for any $k$. However we can
 use the previous proofs to show that $P(k)$ may not
equal $Q(k)$ for all other $k$ due to the non invertibility of $\Phi(m,0)$
in the case of $\Z_+$ and that of $\Phi(0,m)$ in the case of $\Z_-$.

\medskip

\begin{remark}\label{rem3bisadd}
 Assume the hypotheses of Proposition \ref{propnew} (i).
Suppose \eqref{lineq} has an exponential dichotomy with respect to a projection
$Q(k)$ such that $Q(m)=P(m)$. Then if $k>m$, ${\cal R}P(k)={\cal R}Q(k)$ by Lemma \ref{lem0} and
\[ {\cal N}P(k)=\Phi(k,m)({\cal N}P(m))=\Phi(k,m)({\cal N}Q(m))={\cal N}Q(k)\]
so that $P(k)=Q(k)$ for $k\ge m$. For $k<m$, ${\cal R}P(k)={\cal R}Q(k)$
by Lemma \ref{lem0}.
However, from the proof of Proposition \ref{propnew} (i), for any $V\neq {\cal N}P(0)$
such that
$$V \oplus {\cal N} (\Phi(m,0))=  \Phi(m,0)^{-1}({\cal N}P(m))$$
(there exists such a $V$ if $\Phi(m,0)$ is not invertible since, by \eqref{invPk3},
${\cal N}P(0) \oplus {\cal N} (\Phi(m,0))=  \Phi(m,0)^{-1}({\cal N}P(m))$)
we get a unique projection $Q(k)$ satisfying $Q(m)=P(m)$
with respect to which \eqref{lineq} has an exponential dichotomy and such that ${\cal N} Q(0)=V$. So ${\cal N}Q(k)$ is not equal to ${\cal N}P(k)$	for some $k<m$ when $\Phi(m,0)$ is  not invertible.
\end{remark}

\begin{remark}\label{rem4bisadd}
  Assume the hypotheses of Proposition \ref{propnew} (ii). Suppose \eqref{lineq} has an exponential dichotomy with respect to a projection
$Q(k)$ such that $Q(m)=P(m)$. Then if $k<m$, ${\cal N}P(k)={\cal N}Q(k)$ by Lemma \ref{lem0} and by \eqref{invPk2}
\[ {\cal R}P(k)=\Phi(m,k)^{-1}({\cal R}P(m))=\Phi(m,k)^{-1}({\cal R}Q(m))={\cal R}Q(k)\]
so that $P(k)=Q(k)$ for $k\le m$. For $k>m$, ${\cal N}P(k)={\cal N}Q(k)$
by Lemma \ref{lem0}. On the other hand, since ${\cal R}P(m)$ is a complement of ${\cal N}P(m)$
such that, by \eqref{invPk4}, ${\cal N}(\Phi(0,m))\subset {\cal R}P(m)$, then for any linear space $V$ such that
$V \oplus {\cal N}P(0)= \R^n$ and $\Phi(0,m)({\cal R}P(m))\subset V$ (if $\Phi(0,m)$ is
not invertible, there are many such $V\neq {\cal R}P(0)$), there is a unique invariant  projection
$Q(k)$ such that ${\cal N}Q(m)={\cal N}P(m)$, ${\cal R}Q(m)={\cal R}P(m)$ and ${\cal R}Q(0)=V$.
So ${\cal R}Q(k)$ is not equal to ${\cal R}P(k)$ for some $k>m$.
However, note by \eqref{invPk4},
   ${\cal R}P(k) \cap {\cal R}Q(k) \supset {\cal N}(\Phi(0,k))$ for any $0 \le k<m$.
\end{remark}

\section{Extension to a larger interval}

Suppose \eqref{lineq} has an exponential dichotomy on $[m,\infty)$   with $m>0$. In this section we give a criterion ensuring that this dichotomy can be extended to $\Z_+$, that is, \eqref{lineq} has an exponential dichotomy on $\Z_+$ with projections of the same rank.
Similarly, if \eqref{lineq} has an exponential dichotomy on $(-\infty,m]$ with $m<0$, we give a criterion ensuring that this dichotomy can be extended to $\Z_-$, that is, \eqref{lineq} has an exponential dichotomy on $\Z_-$ with projections of the same rank. Note that in the invertible
case these extensions always exist but may not in the noninvertible case. These criteria have already been proved in \cite{PC1} but here we show exactly when the dichotomy can be extended without changing the original projection.		

\medskip

To prove our criterion  for dichotomies on $[m,\infty)$, we need two lemmas.

\begin{lemma}\label{lem1ext} Suppose \eqref{lineq} has an exponential dichotomy on $k\ge 1$ with projection $P(k)$. Then this exponential dichotomy can be extended to $k\ge 0$ with $P(k)$ unchanged for $k\ge 1$ if and only if there exists a splitting $U\oplus V=\R^n$
such that $A(0)(U)\subset {\cal R}P(1)$, $A(0)(V)={\cal N}P(1)$ and $A(0)$ is one to one on $V$.
\end{lemma}

\begin{proof}   Obviously the condition is necessary since we may take $U={\cal R}P(0)$
		and $V={\cal N}P(0)$. To prove the sufficiency, we define
$P(0)$ as the projection with range $U$ and nullspace $V$. Then
since $A(0)(U)\subset {\cal R}P(1)$,
\[ P(1)A(0)P(0)=A(0)P(0).\]
Next $A(0)(V)= {\cal N}P(1)$ means that $A(0)$ maps the nullspace of $P(0)$ one to one
on to the nullspace of $P(1)$. In particular,
\[ P(1)A(0)(I-P(0))=0.\]
Adding the last two equations, we get the invariance
\[ P(1)A(0)=A(0)P(0).\]
Now we know there exist positive constants $K$ and $\alpha$ such that for $k\ge m\ge 1$
\[ |\Phi(k,m)P(m)|\le Ke^{-\alpha(k-m)},\quad
|\Phi(m,k)(I-P(k))|\le Ke^{-\alpha(k-m)}.\]
Then if $k\ge 1$
\[ \begin{array}{rl}
|\Phi(k,0)P(0)|&=|\Phi(k,1)A(0)P(0)|=|\Phi(k,1)P(1)A(0)|
\le Ke^{-\alpha(k-1)}|A(0)|\\ \\
&=K|A(0)|e^{\alpha}e^{-\alpha k}.\end{array}\]
Also
\[ |\Phi(1,0)P(0)|=|A(0)P(0)|=|A(0)P(0)|e^{\alpha}e^{-\alpha(1-0)}\]
and
\[ |\Phi(0,0)P(0)|=|P(0)|=|P(0)|e^{\alpha(0-0)}\]

So for $k\ge m\ge 0$,
\[  |\Phi(k,m)P(m)|\le K_1e^{-\alpha(k-m)},\]
where $K_1=\max\{K,K|A(0)|e^{\alpha},|A(0)||P(0)|e^{\alpha},|P(0)|\}$.
\medskip

Next if $k\ge 1$,
since $\Phi(k,0)=\Phi(k,1)A(0)$, the map $\Phi(k,0):{\cal N}P(0)\to{\cal N}P(k)$ has
an inverse, which we denote by $\Phi(0,k)$, given by $A(0)^{-1}\Phi(1,k)$,
where $A(0)^{-1}$ is the inverse of $A(0):{\cal N}P(0)\to{\cal N}P(1)$. Then
\[\begin{array}{rl}
|\Phi(0,k)(I-P(k))|
&=|A(0)^{-1}\Phi(1,k)(I-P(k))|\le |A(0)^{-1}|Ke^{-\alpha(k-1)}\\ \\
&=|A(0)^{-1}|Ke^{\alpha}e^{-\alpha k}.\end{array}\]
Also
\[ |\Phi(0,1)(I-P(1))| =|A(0)^{-1}(I-P(1))|
\le |A(0)^{-1}(I-P(1))|e^{\alpha}e^{-\alpha(1-0)}\]\
and
\[ |\Phi(0,0)(I-P(0))| =|(I-P(0))|=|(I-P(0))|e^{-\alpha(0-0)}.\]\
So for $k\ge m\ge 0$,
\[  |\Phi(m,k)(I-P(k))|\le K_2e^{-\alpha(k-m)},\]
where $K_2=\max\{K,|A(0)^{-1}|Ke^{\alpha},|A(0)^{-1}(I-P(1))|e^{\alpha},|I-P(0)|\}$.
\medskip

So we can extend the dichotomy to $k\ge 0$  without changing $P(k)$ for $k\ge 1$.

\end{proof}

\begin{lemma}\label{lemext3} Let $S$ be a subspace of $\R^n$ with dim$(S)=r$ and suppose $A$, $B$ are
	$n\times n$ matrices such that
	\[{\rm dim}\,((AB)^{-1}(S))=r.\]
	Then
	\[{\rm dim}\,(A^{-1}(S))=r.\]
\end{lemma}

\begin{proof}  Let $U=A^{-1}(S)$.
		Now $A(U) =S\cap {\cal R}(A)
		\subset S$ so that since ${\cal N}(A)\subset U$
		\[ {\rm dim}\,(S \cap {\cal R}(A))={\rm dim}\,U-{\rm dim}({\cal N}(A))
		\]
		and hence
		\begin{equation}\label{eqq9} {\rm dim}\,(S \cap {\cal R}(A))+  {\rm dim}({\cal N}(A))={\rm dim}(A^{-1}(S)).\end{equation}
	   By exactly the same argument,
	   \begin{equation} \label{eqq8} {\rm dim}\,(S \cap {\cal R}(AB))+  {\rm dim}({\cal N}(AB))={\rm dim}((AB)^{-1}(S))=r.\end{equation}
	   Next let $q= {\rm dim} ({\cal R}(AB))$ and $p= {\rm dim} ({\cal R}(A))$ so that $q \le p \le n$
	and  ${\rm dim}\,({\cal N}(AB))=n-q$, ${\rm dim}\,({\cal N}(A))=n-p$.
	From \eqref{eqq8} we get
	\[
	{\rm dim}\, (S\cap {\cal R}(AB)) = r + q - n.
	\]
	However we know that
	\begin{align*}
	{\rm dim}\, (S\cap {\cal R}(AB))
	&=  {\rm dim}\, (S) + {\rm dim}{\cal R}(AB) -  {\rm dim}\, (S+{\cal R}(AB))\\
	&= r + q - {\rm dim}(S+{\cal R}(AB)).
	\end{align*}
	As a consequence
	\[
	{\rm dim}\, (S+{\cal R}(AB)) = n,
	\]
	that is, $\R^n = S+{\cal R}(AB) \subseteq S+{\cal R}(A) \subseteq\R^n$. Hence $S+{\cal R}(A) = \R^n$ and
	\begin{align*}
	{\rm dim}\, (S\cap {\cal R}(A))
	&=  {\rm dim}\, (S) + {\rm dim}{\cal R}(A) -  {\rm dim}\, (S+{\cal R}(A))\\
	&= r + p - n = r -{\rm dim}({\cal N}(A)).
	\end{align*}
	Then it follows from \eqref{eqq9} that dim\,$A^{-1}(s)=r$ and the proof of the
	lemma is finished.
\end{proof}

Here is our criterion.

\begin{theorem}\label{thm2}
Suppose \eqref{lineq} has an exponential dichotomy on $k\ge m$ with projection
  $P(k)$  of rank $r$, where $m\ge 1$. Then the exponential dichotomy can be extended to $k\ge 0$  with $P(k)$ unchanged for $k\ge m$  if
\begin{equation}\label{eqm} {\rm dim}\,(\Phi(m,0)^{-1}({\cal R}P(m)))=r.\end{equation}
  Conversely, if \eqref{lineq} has an exponential dichotomy on $k\ge 0$ with projection of rank $r$, then \eqref{eqm} holds.
\end{theorem}

\begin{proof} We prove the sufficiency by induction on $m\ge 1$. First we prove it for $m=1$. So we are assuming \eqref{lineq} has an exponential dichotomy on $k\ge 1$ with projection of rank $r$ and that dim$(A(0)^{-1}({\cal R}P(1))=r$.
\medskip

Set $U=A(0)^{-1}({\cal R}P(1))$. Then dim$(U)=r$. Note that
\[{\cal N}(A(0))\subset  A(0)^{-1}({\cal N}P(1)).\]
Let $V$ be a subspace such that
\[ V\oplus {\cal N}(A(0))=A(0)^{-1}({\cal N}P(1)).\]
Then $A(0)$ maps $V$ one to one onto ${\cal N}P(1)$ so that, in particular,
dim $(V)=n-r$. Let $x\in U\cap V$. Then $A(0)x\in {\cal R}P(1)\cap{\cal N}P(1)$
and so $A(0)x=0$. Since $A(0)$ is one to one on $V$, it follows that $x=0$.
Hence $U\cap V=\{0\}$ and so $U\oplus V=\R^n$.
Since we also have $A(0)(U)\subset {\cal R}P(1)$, it follows from Lemma \ref{lem1ext} that the exponential dichotomy can be extended to $k\ge 0$   with $P(k)$ unchanged
	for $k\ge 1$. This proves the case $m=1$.
\medskip

Now we assume the sufficiency is true for $m-1\ge 1$ and
prove it for $m$. So suppose \eqref{lineq} has an exponential dichotomy on $k\ge m$ with projection $P(k)$ of rank $r$ and that \eqref{eqm} holds.
Since $\Phi(m,0)=\Phi(m,1)\Phi(1,0)$, it follows from \eqref{eqm}
and Lemma \ref{lemext3} that dim $\Phi(m,1)^{-1}({\cal R}P(m)))=r$. Then
it follows from the induction hypothesis that
\eqref{lineq} has an exponential dichotomy on $k\ge 1$  with $P(k)$ unchanged for $k\ge m$.
\medskip

 Set $U=A(0)^{-1}({\cal R}P(1))$. Then $U=\Phi(m,0)^{-1}({\cal R}P(m))$ since
\[ \Phi(m,0)x\in {\cal R}P(m) \iff A(0)x \in \Phi(m,1)^{-1}({\cal R}P(m))={\cal R}P(1).\]
Hence dim $(U)=r$. Then we prove as for $m=1$
that the exponential dichotomy can be extended to $k\ge 0$ with the
projection unchanged for $k\ge 1$ and hence for $k\ge m$.  So the sufficiency is proved for $m$.
\medskip

Now we prove the necessity. So suppose   \eqref{lineq} has an exponential dichotomy on $k\ge 0$ with projections $Q(k)$ of rank $r$. Then, by Lemma \ref{lem0}, ${\cal R}Q(m)={\cal R}P(m)$ and so by \eqref{invPk2},
\[ {\rm dim}\,(\Phi(m,0)^{-1}({\cal R}P(m)))={\rm dim}\,(\Phi(m,0)^{-1}({\cal R}Q(m)))
={\rm dim}\,({\cal R}Q(0))=r.\]
 The necessity is proved.
\end{proof}

\begin{remark} The special case $m=1$ of this theorem states that an exponential
dichotomy for \eqref{lineq} on $k\ge 1$ with projection $P(k)$ of rank $r$ can be extended to
$k\ge 0$ if and only if dim $(A(0)^{-1}({\cal R}P(1))=r$. This can be used repeatedly
to find out how far a dichotomy on $k\ge m$ can be extended. It can be extended to
$k\ge m-1$ if and only if dim $(A(m-1)^{-1}({\cal R}P(m))=r$. If this holds,
then ${\cal R}P(m-1)=A(m-1)^{-1}({\cal R}P(m))$ and the dichotomy can be further
extended to $k\ge m-2$ if dim $(A(m-2)^{-1}({\cal R}P(m-1))=r$ and so on.
\end{remark}

Now we consider the case of $\Z_-$.
\begin{theorem} \label{thm3} For $m<0$, suppose \eqref{lineq}
has an exponential dichotomy on $k\le m$ with projection $P(k)$   of rank $r$.
Then
	
{\rm (i)} \eqref{lineq} has an exponential dichotomy on $k\le 0$ with projection of
rank $r$ if and only if $\Phi(0,m)$ is one to one on ${\cal N}P(m)$;

{\rm (ii)} \eqref{lineq} has an exponential dichotomy on $k\le 0$ with projection $P(k)$ for $k\le m$ if and only if $\Phi(0,m)$ is one to one on ${\cal N}P(m)$ and ${\cal N}(\Phi(0,m))\subset {\cal R}P(m)$.
\end{theorem}

\begin{proof} The necessity of the condition  in (i)  is obvious from the definition of exponential dichotomy.   The necessity of the condition in (ii) follows from the definition of exponential dichotomy and from \eqref{invPk4}.
\medskip	

Next we prove the sufficiency for $m=-1$. For this we need a lemma.

\begin{lemma}\label{lemext4}
Suppose \eqref{lineq} has an exponential dichotomy on   $k\le -1$  with invariant family of projections $P(k)$.
 Then the exponential dichotomy can be extended to $\Z_-$ without changing
	the projections for $k\le -1$ if and only if
	there exists a projection $P(0)$ such that $A(-1)P(-1)=P(0)A(-1)$ and $A(-1)$ maps
	${\cal N}P(-1)$ bijectively on to ${\cal N}P(0)$.
\end{lemma}

\begin{proof}  The necessity of the conditions follows from the definition of exponential dichotomy. To prove the sufficiency, we observe first that we know already that $A(k)P(k)=P(k+1)A(k)$ for $k\le -2$. By hypothesis, this also holds for $k=-1$.
 Hence the family $\{P(k)\}_{k\le 0}$ is invariant.
\medskip

  Next we know that $\Phi(k,m)$ maps ${\cal N}P(m)$ bijectively onto ${\cal N}P(k)$ for
	$m<k\le -1$. For $m<0$, we have $\Phi(0,m)=\Phi(0,-1)\Phi(-1,m)=A(-1)\Phi(-1,m)$ and hence
$\Phi(0,m):{\cal N}P(m)\to {\cal N}P(0)$
is invertible because both $\Phi(-1,m):{\cal N}P(m)\to {\cal N}P(-1)$ and $A(-1):{\cal N}P(-1)\to {\cal N}P(0)$ are invertible.
Taking the inverses in the appropriate spaces, from $\Phi(0,m)=A(-1)\Phi(-1,m)$, we get
\begin{equation}\label{phia}
\Phi(m,0)=\Phi(m,-1)A(-1)^{-1}.
\end{equation}

Next we prove the dichotomy inequalities.

i) For $m\le k\le -1$ we know we have
\[
|\Phi(k,m)P(m)|\le Ke^{-\alpha(k-m)}.
\]
Hence we only have to consider the two cases: $m<k=0$ and $m=k=0$. We have
\[\begin{array}{l}
|\Phi(0,0)P(0)|=|P(0)|=|P(0)|e^{-\alpha(0-0)}\\
|\Phi(0,m)P(m)|=|\Phi(0,-1)\Phi(-1,m)P(m)|= |A(-1)\Phi(-1,m)P(m)|\\
\qquad \le |A(-1)|Ke^\alpha e^{-\alpha(0-m)},
\end{array}\]
for any $m<0$.
\medskip

ii) For $m\le k\le -1$ we know we have
\[
|\Phi(m,k)(I-P(k))|\le Ke^{-\alpha(k-m)},
\]
 where $\Phi(m,k)$ is the inverse of $\Phi(k,m):{\cal N}P(m)\to {\cal N}P(k)$.
Hence we only have to consider the two cases: $m<k=0$ and $m=k=0$.
First note
\[ |\Phi(0,0)(I-P(0))|=|I-P(0)|= |I-P(0)|e^{-\alpha(0-0)}.\]
Then,   using \eqref{phia} , for any $x\in{\cal N}P(0)$ and for any $m<0$:
\begin{align*}
|\Phi(m,0)x|
&=|\Phi(m,-1)A(-1)^{-1}x|\\
& = |\Phi(m,-1)(I-P(-1))A(-1)^{-1}x|\quad {\rm since\;} A(-1)^{-1}:{\cal N}P(0)\to {\cal N}P(-1)\\
&\le Ke^{\alpha(m+1)}|A(-1)^{-1}|\, |x|,
\end{align*}
that is,
\[
|\Phi(m,0)(I-P(0))| \le Ke^{\alpha}|A(-1)^{-1}| \, |I-P(0)|e^{-\alpha(0-m)}.
\]
This completes the proof of the lemma.
\end{proof}

  Now we continue the proof of Theorem \ref{thm3}. To prove the sufficiency   of the condition in (i)  for $m=-1$, suppose
$A(-1)$ is one to one on ${\cal N}P(-1)$ and  set $r= \textrm{dim}({\cal N}P(-1))$, $k=\textrm{dim}({\cal N}A(-1))$.
Note that, from \eqref{invPk4} it follows that ${\cal N}A(-1)={\cal N}\Phi(0,-1)$ intersects ${\cal N}P(-1)$ in $\{0\}$. Since, by Proposition \ref{prop1}, ${\cal R}P(-1)$ can be an arbitrary complement of
${\cal N}P(-1)$, we may choose it so that
\[ {\cal N}A(-1) \subset {\cal R}P(-1).\]
Let us set
$$U=  A(-1)({\cal N}P(-1)) \qquad \textrm{and} \qquad V=  A(-1){\cal R}P(-1). $$ Then
$\textrm{dim}(U)=r$ and  $\textrm{dim}(V)=n-r-k$. Also
$U \cap V= \{0\}$ . Indeed if $x=A(-1)v=A(-1)u$, with $u\in{\cal N}P(-1)$ and $v\in {\cal R}P(-1)$, then $u-v\in{\cal N}A(-1)\subset {\cal R}P(-1)$. So $u=v+(u-v)\in{\cal R}P(-1)$ from which we get $u=0$ and also $x=0$.
Hence we can find $W$ such that $\textrm{dim}(W)=k$ and
$$U \oplus V \oplus W = \R^n .$$
Let $P(0)$ be the projection such that ${\cal N}P(0)=U$ and ${\cal R}P(0)=V \oplus W$.
Since $V=  A(-1){\cal R}P(-1)$, then
\[ P(0)A(-1)P(-1)=A(-1)P(-1).\]
Next $A(-1)$  maps the nullspace of $P(-1)$ one to one
on  the nullspace of $P(0)$. In particular,
\[ P(0)A(-1)(I-P(-1))=0.\]
Adding the last two equations, we get the invariance
\[ P(0)A(-1)=A(-1)P(-1).\]
So, by Lemma \ref{lemext4} we can extend the dichotomy to $k\le 0$   without changing the projections
	for $k\le -1$ (after the initial change of $P(-1)$).
  So the sufficiency in (i) is proved for $m=-1$.  The argument for (ii) is similar.
  	The difference is that we know that ${\cal N}A(-1) \subset {\cal R}P(-1)$ so that $P(-1)$
  	does not need to be changed.  	
\medskip

To complete the proof of the  sufficiency of the condition in (i)  by induction on $-m\ge 1$, we suppose  the sufficiency holds  for $-m-1\ge 1$ and prove it for $-m$. So suppose \eqref{lineq} has an exponential dichotomy on $k\le m$ and that $\Phi(0,m)$ is one to one on ${\cal N}P(m)$. Since $\Phi(0,m)=\Phi(0,m+1)A(m)$, it follows that $A(m)$ is one to one on ${\cal N}P(m)$. So by the $m=-1$ case, it follows that the exponential dichotomy can be extended to $k\le m+1$ with ${\cal N}P(m+1)=A(m)({\cal N}P(m))$. However then $\Phi(0,m+1)$ is one to one on ${\cal N}P(m+1)$. By the induction hypothesis it follows that the exponential dichotomy can be extended to $k\le 0$.  This completes the proof of the sufficiency of the condition in (i).
\medskip

To prove the sufficiency of the condition in (ii), suppose \eqref{lineq} has an exponential dichotomy on $k\le m$, that $\Phi(0,m)$ is one to one on ${\cal N}P(m)$ and that ${\cal N}(\Phi(0,m))\subset {\cal R}P(m)$. Then by (i), we know \eqref{lineq} has an exponential dichotomy on $k\le m$ with projection $Q(k)$ having the same rank as $P(k)$. By Proposition \ref{prop1}, we know that ${\cal N}Q(m)={\cal N}P(m)$. Thus ${\cal R}P(m)$ is a complement to ${\cal N}Q(m)$. So by Proposition \ref{propnew}, \eqref{lineq} has an exponential dichotomy on $k\le 0$ with projection equal to $P(m)$ when $k=m$. When $k<m$, the nullspace of the projection is uniquely determined as ${\cal N}P(k)$ and, by \eqref{invPk2}, the range is $\Phi(k,m)^{-1}({\cal R}P(m))$. So $P(k)$ is unchanged for $k<m$ also.
\end{proof}

\begin{remark}\label{rem4}
To extend a dichotomy from $k\le m$, where $m<0$, to $k=0$, we could do
it step by step using the case $m=-1$ of Theorem \ref{thm3}.
\end{remark}

\section{Roughness of exponential dichotomy}

To prove the roughness theorem we need a lemma. This lemma is just Lemma 3
from Zhou et al \cite{ZLZ} with a slight change in notation (actually the lemma is originally from Barreira et al \cite{BSV}).
In the whole section we always assume that $ \displaystyle \sum_{m=a}^{b}\mu_m=0$ if $a>b$.

\begin{lemma}\label{dislem2}
	
	{\rm (i)} Let $a$ be a fixed integer. Suppose
	$\{\mu_k\}_{k=a}^{\infty}$ is a bounded sequence of nonnegative numbers for which there exist positive numbers $D$, $\alpha$ and $\delta$ such that for $k\geq a$
	\[\mu_k\leq De^{-\alpha(k-a)}+\delta\sum_{m=a}^{k-1}e^{-\alpha(k-m-1)}\mu_m
	+\delta\sum_{m=k}^{\infty}e^{-\alpha(m+1-k)}\mu_m.
	\]
	Then if
	\[ \sigma=\delta(1+e^{-\alpha})(1-e^{-\alpha})^{-1}<1,\]
	we have
	\[
	\mu_k\leq \frac{D}{1-\delta e^{-\alpha}/(1-e^{-(\alpha+\beta)})}e^{-\beta(k-a)},\quad k\ge a, \]
	where
	\[\beta=-\log({\rm cosh}\,\alpha-\sqrt{{\rm sinh}^2\alpha-2\delta{\rm sinh}\,\alpha}).\]
	\medskip
	
	{\rm (ii)} Let $b$ be a fixed integer. Suppose
	$\{\mu_k\}_{k=-\infty}^{b}$ is a bounded sequence of nonnegative numbers for which there exist positive numbers $D$, $\alpha$ and $\delta$ such that for $k\leq b$
	\[\mu_k\leq De^{-\alpha(b-k)}+\delta\sum_{m=-\infty}^{k-1}e^{-\alpha(k-m-1)}\mu_m
	+\delta\sum_{m=k}^{b-1}e^{-\alpha(m+1-k)}\mu_m.
	\]
	Then if
	\[ \sigma=\delta(1+e^{-\alpha})(1-e^{-\alpha})^{-1}<1,\]
	we have
	\[
	\mu_k\leq \frac{D}{1-\delta e^{-\gamma}/(1-e^{-(\alpha+\gamma)})}e^{-\gamma(b-k)},\quad k\le b, \]
	where
	\[\gamma=\beta+\log(1+2\delta{\rm sinh}\,\alpha).\]
\end{lemma}

In our roughness theorem, we do not consider additive perturbations of the form $A(k)+B(k)$ because there are already proofs for this case in the literature. As far as we know, this is the first proof for multiplicative perturbations of the form $A(k)[I+B(k)]$.
For the case of $\Z$, the proof below follows the proof in Zhou et al. \cite{ZLZ} with certain modifications. They did not consider the half-axis cases. We show how these cases can be deduced from the $\Z$ case. The reason we follow the proof by Zhou et al. \cite{ZLZ} is that they obtain the best estimates. In particular, they show that the constants in the dichotomy for the perturbed system approach those for the unperturbed system as the perturbation approaches zero. This was also done by Henry \cite{H} but we prefer the proof in Zhou et al. \cite{ZLZ} as it can be used in the continuous time case also.  We also show the continuity of the projection which Zhou et al. \cite{ZLZ} did not show. Note in this theorem we use Definition \ref{EDdef} as our definition of dichotomy.
\medskip

\begin{theorem} \label{roughthm} Suppose \eqref{lineq} has an exponential dichotomy on an interval $J=[a,\infty)$ or $(-\infty,b]$ or $(-\infty,\infty)$ with projection $P(k)$ and constants $K$, $\alpha$ satisfying \eqref{newk1} and \eqref{newk1a}. Then,  if $|B(k)|\le \delta$, where	
\begin{equation}\label{sigeq}
	\rho\delta=K(1-e^{-\alpha})^{-1}(1+e^{-\alpha})\delta<1,
\end{equation}
the perturbed system
\begin{equation}\label{perteq}
	x(k+1)=A(k)[I+B(k)]x(k)
\end{equation}
also has an exponential dichotomy with projection $Q(k)$ of the same rank with constant $L$ and exponent $\beta$, where
\begin{align*}
	\beta&=-\log\left({\rm cosh}\,\alpha-\sqrt{{\rm sinh}^2\alpha-2K\delta{\rm sinh}\,\alpha}\right),\\
	L&=K(1+K\delta(1-\rho\delta)^{-1}(1-e^{-\alpha})^{-1})\max\{D_1,D_2\},\\	
	D_1&=(1-K(1-e^{-(\alpha+\beta)})^{-1}\delta)^{-1},\\
	D_2&=(1-Ke^{\alpha-\gamma}(1-e^{-(\alpha+\gamma)})^{-1}\delta)^{-1},\\
	\gamma&=\beta+\log(1+2Ke^{\alpha}\delta{\rm sinh}\,\alpha),
\end{align*}
such that
\[
	|Q(k)-P(k)|\le KL(1-e^{-(\alpha+\beta)})^{-1}(1+e^{-(\alpha+\beta)})\delta.
\]
\end{theorem}

\begin{proof} First we prove the theorem for $J=\Z$. We denote by $\Phi(k,m)$, $k\ge m$, the transition matrix for \eqref{lineq}. Note in this proof, whenever $k<m$, $\Phi(k,m)$ means the inverse of $\Phi(m,k):{\cal N}P(k)\to {\cal N}P(m)$. The proof will follow from four claims.

\medskip

{\bf Claim 1:} {\it If $u(k)$ is a solution of
	\begin{equation}\label{inhomog}
	x(k+1)=A(k)(I+B(k))x(k)+f(k),
	\end{equation}	
	on an interval $[a,b]$, then
	\begin{equation} \label{ab3}\begin{array}{rl}
	&u(k)\\ \\
	& =\displaystyle\Phi(k,a)P(a)u(a) +\Phi(k,b)(I-P(b))u(b)\\ \\
	&\displaystyle+\sum_{m=a}^{k-1}\Phi(k,m)P(m)B(m)u(m)
	+\sum_{m=a}^{k-1}\Phi(k,m+1)P(m+1)f(m)\\ \\
	&\displaystyle-\sum_{m=k}^{b-1}\Phi(k,m)(I-P(m))B(m)u(m)
	-\sum_{m=k}^{b-1}\Phi(k,m+1)(I-P(m+1))f(m)
	\end{array}\end{equation}
	gives a representation of the solution $u(k)$ in terms of its boundary values $P(a)u(a)$ and $(I-P(b))u(b)$.}

\begin{proof} If $u(k)$ is a solution of \eqref{perteq}, then it is easy to show that for $k\ge a$,
	\[ u(k)=\Phi(k,a)u(a) +\sum^{k-1}_{m=a}\Phi(k,m)B(m)u(m)
	+\sum^{k-1}_{m=a}\Phi(k,m+1)f(m), \]
	where $\Phi(k,m)$ is the transition matrix for \eqref{lineq}, from which it follows that
	\begin{equation} \label{ab1}
	\begin{array}{rl}
	P(k)u(k)
	&=\displaystyle\Phi(k,a)P(a)u(a) +\sum_{m=a}^{k-1}\Phi(k,m)P(m)B(m)u(m)\\ \\
	&\qquad\qquad +\displaystyle\sum_{m=a}^{k-1}\Phi(k,m+1)P(m+1)f(m).
	\end{array} \end{equation}
	Next we show on $k\le b$ that
	\begin{equation} \label{ab2}\begin{array}{rl}(I-P(k))u(k)
	&=\displaystyle\Phi(k,b)(I-P(b))u(b)-\sum_{m=k}^{b-1}\Phi(k,m)(I-P(m))B(m)u(m)\\ \\
	&\qquad\qquad -\displaystyle\sum_{m=k}^{b-1}\Phi(k,m+1)(I-P(m+1))f(m).
	\end{array}\end{equation}
	First  we have for $k\le b$
	\[ u(b)=\Phi(b,k)u(k) +\sum^{b-1}_{m=k}\Phi(b,m)B(m)u(m)
	+\sum^{b-1}_{m=k}\Phi(b,m+1)f(m). \]
	Multiply by $I-P(b)$ to get, using the invariance,
	\begin{align*}
	(I-P(b))u(b)
	&=\Phi(b,k)(I-P(k))u(k) +\sum^{b-1}_{m=k}\Phi(b,m)(I-P(m))B(m)u(m)\\
	&\qquad\qquad +\sum^{b-1}_{m=k}\Phi(b,m+1)(I-P(m+1))f(m). \end{align*}
	Next multiply by $\Phi(k,b)$ the inverse of $\Phi(b,k):{\cal N}P(k)\to{\cal N}P(b)$
	to get
	\begin{align*}\Phi(k,b)(I-P(b))u(b)
	&=(I-P(k))u(k) +\sum^{b-1}_{m=k}\Phi(k,m)(I-P(m))B(m)u(m)\\
	&\qquad\qquad +\sum^{b-1}_{m=k}\Phi(k,m+1)(I-P(m+1))f(m) \end{align*}
	and hence \eqref{ab2}. Here we have used the fact that $\Phi(k,b)\Phi(b,m)=\Phi(k,m)$.
	This follows from $\Phi(b,k)=\Phi(b,m)\Phi(m,k)$ if $k\le m\le b$, where
	we restrict the mappings to the nullspaces of the respective projections. We take
	inverses to get $\Phi(k,b)=\Phi(k,m)\Phi(m,b)$ and then multiply both sides
	from the right by $\Phi(b,m)$.
	\medskip
	
	Adding \eqref{ab1} and \eqref{ab2}, we obtain \eqref{ab3}.
\end{proof}

{\bf Claim 2: } {\it Let $f(k)$ be a sequence bounded on $\Z$. Then the equation
	\begin{equation}\label{eqf} x(k+1)=A(k)[I+B(k)]x(k)+f(k)\end{equation}
	has a unique bounded solution $u(k)$ which satisfies}
\begin{equation}\label{eqbf}\begin{array}{rl} &u(k)\\ \\
&=\displaystyle\sum_{m=-\infty}^{k-1}\Phi(k, m)P(m)B(m)u(m)-\sum_{m=k}^{\infty}\Phi(k,m)(I-P(m))B(m)u(m)\\ \\
&+\displaystyle\sum_{m=-\infty}^{k-1}\Phi(k, m+1)P(m+1)f(m)-\sum_{m=k}^{\infty}\Phi(k,m+1)(I-P(m+1))f(m).\end{array}\end{equation}

\begin{proof} Let $u(k)$ be a bounded solution of \eqref{eqf}.
	  Then if $a<b$, $u(k)$ satisfies \eqref{ab3} for $a\le k\le b$.
	Note that
	\[
	|\Phi(k,a)P(a)u(a)|\le Ke^{-\alpha(k-a)}|u(a)|\to 0
	\] as $a\to -\infty$
	and
	\[|\Phi(k,b)(I-P(b))u(b)|\le Ke^{-\alpha(b-k)}|u(b)|\to 0\]
	as $b\to \infty$. Also the norms of the four summands   in \eqref{ab3}  are bounded by
	\[Ke^{-\alpha(k-m)}\delta\|u\|,\; Ke^{-\alpha(k-m-1)}\|f\|,\;
	Ke^{-\alpha(m-k)}\delta\|u\|,\; Ke^{-\alpha(m+1-k)}\|f\|,
	\]
	respectively, where $\|\cdot\|$ denotes the supremum norm. Let $a\to-\infty$ and $b\to\infty$   in \eqref{ab3}.
Then the sums converge to the infinite sums and \eqref{eqbf} is obtained.
	\medskip
	
	All that remains to show is that \eqref{eqbf} has a unique bounded solution
	which is also a solution of \eqref{eqf}.
	To this end, let $X$ be the Banach space of bounded sequences
	${\bf u}=\{ u(k)\}_{k=-\infty}^{\infty}$ with norm
	\[ \|{\bf u}\| = \sup_{k\in\Z}|u(k)|. \]
	We define the operator $T:X\to X$ by
	\[\begin{array}{rl}  (T{\bf u})_{k}
	&=\displaystyle\sum_{m=-\infty}^{k-1}\Phi(k, m)P(m)B(m)u(m)-\sum_{m=k}^{\infty}\Phi(k,m)(I-P(m))B(m)u(m)\\ \\
	&+\displaystyle\sum_{m=-\infty}^{k-1}\Phi(k, m+1)P(m+1)f(m)-\sum_{m=k}^{\infty}\Phi(k,m+1)(I-P(m+1))f(m).\end{array}\]
	We see that
	\[\begin{array}{rl}  |(T{\bf u})_{k}|
	&\le \displaystyle\left[\sum_{m=-\infty}^{k-1}Ke^{-\alpha(k-m)}+
	\sum_{m=k}^{\infty}Ke^{-\alpha(m-k)}\right]\delta\|u\|\\ \\
	&\displaystyle +\left[\sum_{m=-\infty}^{k-1}Ke^{-\alpha(k-m-1)}+
	\sum_{m=k}^{\infty}Ke^{-\alpha(m-k+1)}\right]\|f\|\\ \\
	&\le \rho(\delta\|u\|+\|f\|)
	\end{array}\]
	so that $T{\bf u}\in X$.
	Next if ${\bf u}= \{u(k)\}^{\infty}_{k=-\infty}$
	and ${\bf v}= \{v(k)\}^{\infty}_{k=-\infty}$ are in $X$, then for all $k$
	\[ \begin{array}{rl}
	&|(T{\bf u})_{k}-(T{\bf v})_{k}|\\ \\
	& = \left|\displaystyle \sum_{m=-\infty}^{k-1}\Phi(k,m)P(m)B(m)(u(m)-v(m))
	\right. \\ \\
	&  \qquad\qquad -\left.\displaystyle\sum_{m=k}^{\infty}\Phi(k, m)(I-P(m))B(m)
	(u(m)- v(m)) \right| \\ \\
	& \leq  \displaystyle \sum_{m=-\infty}^{k-1}Ke^{-\alpha(k-m)}
	\delta|u(m)-v(m)|
	+ {\displaystyle  \sum_{m=k}^{\infty}}Ke^{-\alpha(m-k)}\delta|u(m)-v(m)|. \end{array}\]
	Hence
	\[
	\|T{\bf u}-T{\bf v}\|\leq \rho\delta\|{\bf u}-{\bf v}\|.
	\]
	Since $\rho\delta< 1$, it follows that $T$ is a contraction and hence has a unique fixed point ${\bf u}= \{u(k)\}^{\infty}_{k=-\infty}$ which is the unique
	bounded solution of \eqref{eqbf}.
	Then
	\[\begin{array}{rl} &u(k+1)\\ \\
	&=\displaystyle\sum_{m=-\infty}^{k}\Phi(k+1, m)P(m)B(m)u(m)-\sum_{m=k+1}^{\infty}\Phi(k+1,m)(I-P(m))B(m)u(m)\\ \\
	&+\displaystyle\sum_{m=-\infty}^{k}\Phi(k+1, m+1)P(m+1)f(m)-\sum_{m=k+1}^{\infty}\Phi(k+1,m+1)(I-P(m+1))f(m)\\ \\
	&= A(k)u(k)+A(k)P(k)B(k)u(k)+A(k)(I-P(k))B(k)u(k) \\ \\
	&\qquad\qquad +P(k+1)f(k)+(I-P(k+1))f(k)\\ \\
	&= A(k)u(k)+A(k)B(k)u(k)+f(k). \end{array} \]
	So $u(k)$ is a solution of \eqref{eqf}.
\end{proof}

{\bf Claim 3:} {\it For given $a\in\Z$ and a vector
	$\xi \in {\cal R}(P(a))$, there is a unique solution $u(k)$ of \eqref{perteq}
	bounded in $k\ge a$ such that
	\begin{equation} \label{BC2} P(a)u(a)= \xi.
	\end{equation}
	In fact,
	\begin{equation}\label{ueqplus}\begin{array}{rl}
	u(k) & =\Phi(k, a)\xi
	+ \displaystyle\sum_{m=a}^{k-1}\Phi(k,m)P(m)B(m)u(m) \\ \\
	& \qquad\qquad  -\displaystyle \sum_{m=k}^{\infty}\Phi(k, m)(I-P(m))B(m)u(m)
	\end{array}\end{equation}
	for $k \ge a$.}
\begin{proof} By a similar argument to that used in the proof of Claim 2 (here we need only let $b\to\infty$), we show that if $u(k)$ is a bounded solution of \eqref{perteq} satisfying \eqref{BC2}, then $u(k)$ must satisfy \eqref{ueqplus}. Next we show \eqref{ueqplus} has a unique solution $u(k)$ bounded in $k\ge a$. To this end, let $X$ be the Banach space of sequences
	${\bf u}=\{ u(k)\}^{\infty}_{k=a}$ with norm
	\[ \|{\bf u}\| = \sup_{k\ge a}|u(k)|. \]
	We define the operator $T:X\to X$ by
	\[\begin{array}{rl}  (T{\bf u})_{k}
	&=\displaystyle\Phi(k, a)\xi
	+ \sum_{m=a}^{k-1}\Phi(k,m)P(m)B(m)u(m) \\ \\
	& \qquad\qquad\displaystyle -\sum_{m=k}^{\infty}\Phi(k,m)(I-P(m))B(m)u(m)
	\end{array} \]
	for $k\geq a$. Arguing as in Claim 2 we see that   $T$ maps $X$
		into itself
		and that if ${\bf u}= \{u(k)\}^{\infty}_{k=a}$
		and ${\bf v}= \{v(k)\}^{\infty}_{k=a}$ are in $X$, then
		\[ \|T{\bf u}-T{\bf v}\|\leq \rho\delta\|{\bf u}-{\bf v}\|. \]
	  Since $\rho\delta< 1$, it follows that $T$ is a contraction and hence has a
	unique fixed point ${\bf u}= \{u(k)\}^{\infty}_{k=a}$ which satisfies \eqref{ueqplus}
	for $k \geq a$. Clearly \eqref{BC2} holds and,
		by a similar argument to that used in Claim 2,  for  $k \geq a$ we show that
		\[u(k+1)= A(k)u(k)+A(k)B(k)u(k).\]
	So $u(k)$ is indeed a solution of \eqref{perteq} satisfying \eqref{BC2}. On the
	other hand, it follows from \eqref{ueqplus} that any such solution must be a fixed
	point of $T$. Thus the uniqueness is established.
\end{proof}

{\bf Claim 4:} {\it Given $b\in\Z$ and a vector
	$\eta \in {\cal N}(P(b))$, there is a unique solution $u(k)$ of \eqref{perteq}
	bounded in $k\le b$ such that
	\begin{equation} \label{BC3} (I-P(b))u(b)= \eta.
	\end{equation}
	Moreover,
	\begin{equation}\label{ueqminus}\begin{array}{rl}
	u(k) & =\Phi(k, b)\eta
	+ \displaystyle\sum_{m=-\infty}^{k-1}\Phi(k,m)P(m)B(m)u(m) \\ \\
	& \qquad\qquad  -\displaystyle \sum_{m=k}^{b-1}\Phi(k, m)(I-P(m))B(m)u(m)
	\end{array}\end{equation}
	for $k \leq b$.}

\begin{proof} By a similar argument to that used in the proof of Claim 2
	(here we need only let $a\to -\infty$), we show that if $u(k)$ is a bounded solution
	of \eqref{perteq} satisfying \eqref{BC3}, then $u(k)$ must satisfy \eqref{ueqminus}.
	Next we show \eqref{ueqminus} has a unique solution $u(k)$ bounded in $k\le b$.
	To this end, let $X$ be the Banach space of sequences
	${\bf u}=\{ u(k)\}_{-\infty}^{b}$ with norm
	\[ \|{\bf u}\| = \sup_{k\leq b}|u(k)|. \]
	We define the operator $T:X\to X$ by
	\[\begin{array}{rl}  (T{\bf u})_{k}
	&=\displaystyle\Phi(k, b)\eta
	+ \sum_{m=-\infty}^{k-1}\Phi(k,m)P(m)B(m)u(m) \\ \\
	& \qquad\qquad\displaystyle -\sum_{m=k}^{b-1}\Phi(k,m)(I-P(m))B(m)u(m)
	\end{array} \]
	for $k\leq b$.   By similar arguments to those used in the
		proofs of the previous two claims, we show that $T$ has a unique fixed
		${\bf u}= \{u(k)\}_{k=-\infty}^{b}$ for which \eqref{BC3} holds and that for $k\le b-1$
		\[u(k+1)= A(k)u(k)+A(k)B(k)u(k).\]
	 So $u(k)$ is indeed a solution of \eqref{perteq} satisfying \eqref{BC3}. On the
	other hand, it follows from \eqref{ueqminus} that any such solution must be a fixed
	point of $T$. Thus the uniqueness is established.
\end{proof}

Now we complete the proof of Theorem \ref{roughthm}.
\medskip

{\bf Definition of the projections and proof of their invariance properties:}

Let $\Psi(k,m)$ be the transition matrix for \eqref{perteq}. For $m\in\Z$, we define
\[ V^+(m)=\left\{\xi:\sup_{k\ge m}|\Psi(k,m)\xi|<\infty\right\},\]
and $V^-(m)$ as the subspace of those $\xi$ for which there exists a bounded solution $y(k)$ of \eqref{perteq}
on $k\le m$ for which $y(m)=\xi$. It is clear that for $k\ge m$
\[\Psi(k,m)(V^+(m))\subset V^+(k),\quad \Psi(k,m)(V^-(m))\subset V^-(k).\]
Note it follows from Claim 2 with $f=0$ that $0$ is the unique bounded solution
of \eqref{perteq}. Hence
\[ V^+(m)\cap V^-(m) =\{0\}.\]
By Claims 3 and 4, the dimension of $V^+(m)$ is the rank of $P(m)$ and the
dimension of $V^-(m)$ is $n$ minus the rank of $P(m)$. So
\[ V^+(m)\oplus V^-(m) =\R^n.\]
This implies that for $k\ge m$, $\Psi(k,m)$ is one to one on $V^-(m)$ for
if $\Psi(k,m)\xi=0$ for some $\xi\in V^-(m)$, then $\xi\in V^+(m)$ also
and so $\xi=0$. It follows that for $k\ge m$,
\[ \Psi(k,m)(V^-(m))= V^-(k).\]
Now let $Q(m)$ be the projection with range $V^+(m)$ and nullspace $V^-(m)$.
 From the invariance properties for $V^+(m)$ and $V^-(m)$, it follows that
$\Psi(k,m)Q(m)=Q(k)\Psi(k,m)$ for $k\ge m$ and that $\Psi(k,m)$ maps
${\cal N}Q(m)$ bijectively on to ${\cal N}Q(k)$ if $k\ge m$.
\medskip

{\bf Dichotomy inequalities:} As a first step we prove the inequalities
\begin{equation}\label{ineq1}
|\Psi(k,m)Q(m)|\le (1-\rho\delta)^{-1}K,\quad k\ge m
\end{equation}
and
\begin{equation}\label{ineq2}
|\Psi(k,m)(I-Q(m))|\le (1-\rho\delta)^{-1}K,\quad k\le m-1.
\end{equation}

To prove these inequalities, we define for given $m$ and $\xi$,
\[
f(k)=\begin{cases} \xi/K & k=m-1\\ 0 & k\neq m-1.\end{cases}
\]
By Claim 2, \eqref{eqf} has a unique bounded solution $x(k)$ which satisfies
\begin{align*}
x(k)&=\sum^{k-1}_{p=-\infty}\Phi(k,p)P(p)B(p)x(p)
+\sum^{k-1}_{p=-\infty}\Phi(k,p+1)P(p+1)f(p)\\
&\qquad-
\sum^{\infty}_{p=k}\Phi(k,p)(I-P(p))B(p)x(p)
-\sum^{\infty}_{p=k}\Phi(k,p+1)(I-P(p+1))f(p)\\
&=\sum^{\infty }_{p=-\infty}G(k,p)B(p)x(p)+G(k,m)\xi/K,
\end{align*}
where
\[ G(k,p)=\begin{cases}\Phi(k,p)P(p) & k>p\\
-\Phi(k,p)(I-P(p)) & k\le p.\end{cases}
\]
Then for all $k$
\[|x(k)| \le \sum^{\infty}_{p=-\infty}Ke^{-\alpha|k-p|}\delta|x(p)|
+ K^{-1}Ke^{-\alpha|k-m|}|\xi|\le \rho\delta\|x\|+|\xi|,
\]
where $\|x\|=\sup_{k\in\Z}|x(k)|$. It follows that
\begin{equation}\label{eq4} \|x\|\le (1-\rho\delta)^{-1}|\xi|.\end{equation}
Note that $x(k)$ is a bounded solution of \eqref{perteq} for $k\le m-2$ so that $x(m-1)\in V^-(m-1)$. Then
\begin{equation}\label{eq3} x(m)=[A(m-1)+B(m-1)]x(m-1)+\xi/K,\end{equation}
where $[A(m-1)+B(m-1)]x(m-1)=\Psi(m,m-1)x(m-1)\in V^-(m)$. Next $x(k)$ is a bounded solution
of \eqref{perteq} for $k\ge m$ so that $x(m)\in V^+(m)$. Now, from \eqref{eq3},
\[ x(m)=Q(m)x(m)=K^{-1}Q(m)\xi.\]
Then if $k\ge m$,
\[ x(k)=\Psi(k,m)x(m)=K^{-1}\Psi(k,m)Q(m)\xi.\]
It follows from this and \eqref{eq4} that for $k\ge m$
\[ K^{-1}|\Psi(k,m)Q(m)\xi| \le (1-\rho\delta)^{-1}|\xi|.\]
Since $\xi$ is arbitrary, \eqref{ineq1} follows.
\medskip

To prove \eqref{ineq2}, multiply \eqref{eq3} by $I-Q(m)$. Then
\[ 0=(I-Q(m))\Psi(m,m-1)x(m-1)+(I-Q(m))\xi/K\]
so that, by invariance,
\[0=\Psi(m,m-1)(I-Q(m-1))x(m-1)+(I-Q(m))\xi/K.\]
Hence since $x(m-1)\in V^-(m-1)$,
\[0=\Psi(m,m-1)x(m-1)+(I-Q(m))\xi/K.\]
So, since $(I-Q(m))\xi/K \in V^-(m)$ and $x(m-1)\in V^-(m-1)$,
\[ x(m-1)=-\Psi(m-1,m)(I-Q(m))\xi/K.\]
Now for $k\le m-1$, $x(k)\in V^-(k)$ because $x(k)$ is the value at $k$ of a bounded solution on $(-\infty,k]$
of \eqref{perteq}. Since $x(m-1)=\Psi(m-1,k)x(k)\in V^-(m-1)$, it follows that
\[ x(k)=\Psi(k,m-1)x(m-1)=-\Psi(k,m-1)\Psi(m-1,m)(I-Q(m))\xi/K.\]
Since, when restricted to the respective nullspaces, $\Psi(k,m-1)\Psi(m-1,m)$
is the inverse of $\Psi(m,k)$, we have
\[ x(k)=-\Psi(k,m)(I-Q(m))\xi/K.\]
It follows that for $k\le m-1$,
\[|\Psi(k,m)(I-Q(m))\xi|/K \le (1-\rho\delta)^{-1}|\xi|\]
and \eqref{ineq2} follows.
\medskip

Now we prove the dichotomy inequalities. Since $\Psi(k,m)Q(m)\xi$ is a bounded
solution of \eqref{perteq} in $k\ge m$, it follows from Claim 3 that
\begin{align}\label{dich1}
&\Psi(k,m)Q(m)\xi \nonumber\\
&=\Phi(k,m)P(m)Q(m)\xi+\sum^{k-1}_{p=m}\Phi(k,p)P(p)B(p)\Psi(p,m)Q(m)\xi\\
&\qquad-
\sum^{\infty}_{p=k}\Phi(k,p)(I-P(p))B(p)\Psi(p,m)Q(m)\xi. \nonumber
\end{align}
and since $\Psi(k,m)(I-Q(m))\xi$ is a bounded
solution of \eqref{perteq} in $k\le m$, it follows from Claim 4 that
\begin{align}\label{dich2}
\Psi(k,m)(I-Q(m))\xi &=\Phi(k,m)(I-P(m))(I-Q(m))\xi \nonumber\\
&\quad-\sum^{m-1}_{p=k}\Phi(k,p)(I-P(p))B(p)\Psi(p,m)(I-Q(m))\xi\\
&\qquad+
\sum^{k-1}_{p=-\infty}\Phi(k,p)P(p)B(p)\Psi(p,m)(I-Q(m))\xi.\nonumber
\end{align}
Taking $k=m$ in \eqref{dich2}, we get
\begin{align*}
(I-Q(m))\xi
&=(I-P(m))(I-Q(m))\xi\\
&+
\sum^{m-1}_{p=-\infty}\Phi(m,p)P(p)B(p)\Psi(p,m)(I-Q(m))\xi
\end{align*}
from which it follows that
\[P(m)Q(m)\xi=P(m)\xi-\sum^{m-1}_{p=-\infty}\Phi(m,p)P(p)B(p)\Psi(p,m)(I-Q(m))
\xi.
\]
If we substitute this into \eqref{dich1}, we obtain
\begin{align}\label{dich3}
&\Psi(k,m)Q(m)\xi \nonumber\\
&=\Phi(k,m)P(m)\xi-\sum^{m-1}_{p=-\infty}\Phi(k,p)P(p)B(p)\Psi(p,m)(I-Q(m))\xi
\nonumber\\
&\qquad+\sum^{k-1}_{p=m}\Phi(k,p)P(p)B(p)\Psi(p,m)Q(m)\xi\\
&\qquad\qquad-
\sum^{\infty}_{p=k}\Phi(k,p)(I-P(p))B(p)\Psi(p,m)Q(m)\xi. \nonumber
\end{align}
Since this holds for any $\xi \in \R^n$, we find that for $k\ge m$,
\begin{align*}
|\Psi(k,m)Q(m)|
&\le Ke^{-\alpha(k-m)}+\sum^{m-1}_{p=-\infty}Ke^{-\alpha(k-p)}\delta |\Psi(p,m)(I-Q(m))|\\
&\qquad +\sum^{\infty}_{p=m}Ke^{-\alpha|k-p|}\delta|\Psi(p,m)Q(m)|\\
&\le Ke^{-\alpha(k-m)}+\sum^{m-1}_{p=-\infty}Ke^{-\alpha(k-p)}\delta (1-\rho\delta)^{-1}K\\
&\qquad +\sum^{\infty}_{p=m}Ke^{-\alpha|k-p|}\delta|\Psi(p,m)Q(m)|\quad
{\rm using\;} \eqref{ineq2}\\
&\le Ke^{-\alpha(k-m)}+K^2\delta(1-\rho\delta)^{-1}e^{-\alpha}(1-e^{-\alpha})^{-1}e^{-\alpha(k-m)}\\
&\qquad+\sum^{\infty}_{p=m}Ke^{-\alpha|k-p|}\delta|\Psi(p,m)Q(m)|.
\end{align*}
Hence $u_k=|\Psi(k,m)Q(m)|$ satisfies
\begin{align*} u_k
&\le K[1+K\delta(1-\rho\delta)^{-1}(e^{\alpha}-1)^{-1}]e^{-\alpha(k-m)}\\
&\qquad +Ke^{\alpha}\delta\left[\sum^{k-1}_{p=m}e^{-\alpha(k-p-1)}u_m
+\sum^{\infty}_{p=k}e^{-\alpha(p+1-k)}u_m\right].
\end{align*}
By Lemma \ref{dislem2} (i), it follows that for $k\ge m$,
\[
u_k\le K_1e^{-\beta(k-m)},
\]
where
\[
K_1=K(1+K\delta(1-\rho\delta)^{-1}(e^{\alpha}-1)^{-1})D_1
\]
and hence that for $k\ge m$,
\begin{equation}\label{dicheq1} |\Psi(k,m)Q(m)|\le K_1e^{-\beta(k-m)}.\end{equation}
\medskip

Now we prove a similar inequality for $|\Psi(k,m)(I-Q(m))|$ when $k\le m$. To this end, take $k=m$ in \eqref{dich1} to get
\[Q(m)\xi=P(m)Q(m)\xi-
\sum^{\infty}_{p=m}\Phi(m,p)(I-P(p))B(p)\Psi(p,m)Q(m)\xi
\]
so that
\[
(I-P(m))Q(m)\xi=-
\sum^{\infty}_{p=m}\Phi(m,p)(I-P(p))B(p)\Psi(p,m)Q(m)\xi.
\]
Then using this in \eqref{dich2}
\begin{align*}
\Psi(k,m)(I-Q(m))\xi
&=\Phi(k,m)(I-P(m))\xi\\
&\quad+\sum^{\infty}_{p=m}\Phi(k,p)(I-P(p))B(p)\Psi(p,m)Q(m)\xi\\
&\quad\quad-\sum^{m-1}_{p=k}\Phi(k,p)(I-P(p))B(p)\Psi(p,m)(I-Q(m))\xi\\
&\quad\quad+
\sum^{k-1}_{p=-\infty}\Phi(k,p)P(p)B(p)\Psi(p,m)(I-Q(m))\xi.
\end{align*}
So using  \eqref{ineq1},
\begin{align*}
|\Psi(k,m)(I-Q(m))|
&\le Ke^{-\alpha(m-k)}+K\delta\sum^{\infty}_{p=m}e^{-\alpha(p-k)}(1-\rho\delta)^{-1}K\\
&\qquad+Ke^{\alpha}\delta\sum^{m-1}_{p=k}e^{-\alpha(p+1-k)}|\Psi(p,m)(I-Q(m))|\\
&\qquad\qquad+
Ke^{\alpha}\delta\sum^{k-1}_{p=-\infty}e^{-\alpha(k-p-1)}|\Psi(p,m)(I-Q(m))|.
\end{align*}
So with $u_k=|\Psi(k,m)(I-Q(m))|$, we have
\begin{align*}
u_k
&\le K\left[1+K(1-\rho\delta)^{-1}(1-e^{-\alpha})^{-1}\delta\right]e^{-\alpha(m-k)}\\
&\qquad +Ke^{\alpha}\delta \sum^{m-1}_{p=k}e^{-\alpha(p-k+1)}u_p+Ke^{\alpha}\delta\sum^{k-1}_{p=-\infty}e^{-\alpha(k-p-1)}u_p.
\end{align*}
Then by Lemma \ref{dislem2} (ii),
we get
\[ |u_k|\le K_2e^{-\gamma(m-k)},\quad k\le m,\]
where
\[ K_2=K(1+K\delta(1-\rho\delta)^{-1}(1-e^{-\alpha})^{-1})D_2.\]
Hence
\[ |\Psi(k,m)(I-Q(m))|\le K_2e^{-\gamma(m-k)},\quad k\le m.\]
Since $\gamma\ge \beta$ and $(e^{\alpha}-1)^{-1}<(1-e^{-\alpha})^{-1}$, we conclude from the last inequality
and \eqref{dicheq1} that for $k\ge m$,
\[
|\Psi(k,m)Q(m)|\le Le^{-\beta(k-m)},\quad |\Psi(m,k)(I-Q(k))|\le Le^{-\beta(k-m)}.
\]

\medskip

{\bf Continuity of projection:}
Put $k=m$ in \eqref{dich3} to get
\begin{align*}
Q(m)\xi-P(m)\xi
&=-\sum^{m-1}_{p=-\infty}\Phi(m,p)P(p)B(p)\Psi(p,m)(I-Q(m))\xi\\
&\qquad-
\sum^{\infty}_{p=m}\Phi(m,p)(I-P(p))B(p)\Psi(p,m)Q(m)\xi
\end{align*}
so that
\begin{align*}
&|Q(m)-P(m)|\\
&\le \sum^{m-1}_{p=-\infty}Ke^{-\alpha(m-p)}\delta Le^{-\beta(m-p)}+\sum^{\infty}_{p=m}Ke^{-\alpha(p-m)}\delta Le^{-\beta(p-m)}\\
&= KL(1+e^{-(\alpha+\beta)})(1-e^{-(\alpha+\beta)})^{-1}\delta.
\end{align*}
This completes the proof of the theorem for the case $J=\Z$.
\medskip

{\bf The cases $[a,\infty)$ and $(-\infty,b]$:}

We deduce these two cases from the $\Z$ case using the following lemma,
 which generalizes Lemma 5 in Zhou and Zhang \cite{ZZ} to the noninvertible case.
\begin{lemma}\label{extending} Let $J\subset\Z$ be an interval in $\Z$ and assume
	\begin{equation}\label{eq}x(k+1)=A(k)x(k),\quad x\in\R^n\end{equation}
	has an exponential dichotomy on $J$ with constants $K$ and $\alpha$
	and projection $P(k)$. Then there exists
	\begin{equation}\label{eqext}x(k+1)=\tilde A(k)x(k),\quad k\in\Z,
	\end{equation}
	such that
	\[
	\tilde A(k)=A(k), \quad \hbox{ for $k\in J$}
	\]
	and \eqref{eqext} has an exponential dichotomy on $\Z$ with the same constant $K$ and exponent $\al$.
\end{lemma}
\begin{proof}
	We assume \eqref{eq} has an exponential dichotomy on $J=[a,b]$, where $-\infty\le a < b \le \infty$ with constants $K$ and $\alpha$ and projection $P(k)$.
	Of course if, for example $a=-\infty$ instead of $[a,b]$ we consider $(-\infty,b]$ and similarly when $b=\infty$. However we prove the result when $a,b\in\Z$ since when $a=-\infty$ or $b=\infty$ the proof has only to be modified in a trivial way.
	\medskip
	
	For $k\le a-1$, we define
	\[
	\tilde A(k) ={\cal A}=e^{-\alpha}P(a)+e^{\alpha}(I-P(a))\]
	and for $k\ge b$ we define
	\[
	\tilde A(k) ={\cal B}=e^{-\alpha}P(b)+e^{\alpha}(I-P(b)).\]
	Note that ${\cal A}$ and ${\cal B}$ are invertible and
	\[ {\cal A}P(a)=P(a){\cal A},\quad {\cal B}P(b)=P(b){\cal B}.\]
	Also for all integers $k$,
	\[ {\cal A}^k=e^{-\alpha k}P(a)+e^{\alpha k}(I-P(a))\]
	so that for $k\ge 0$
	\[ |{\cal A}^kP(a)|=e^{-\alpha k}|P(a)| \le Ke^{-\alpha k}\]
	and for $k\le 0$
	\[ |{\cal A}^k(I-P(a))|=e^{\alpha k}|I-P(a)| \le Ke^{\alpha k}.\]
	This means that $x(k+1)=\tilde A(k)x(k)$ has an exponential dichotomy on
	$k\le a$ with constant projection $P(a)$ and constants $K$, $\alpha$.
	Similarly for $k\ge 0$
	\[ |{\cal B}^kP(b)|=e^{-\alpha k}|P(b)| \le Ke^{-\alpha k}\]
	and for $k\le 0$
	\[ |{\cal B}^k(I-P(b))|=e^{\alpha k}|I-P(b)| \le Ke^{\alpha k}.\]
	This means that $x(k+1)=\tilde A(k)x(k)$ has an exponential dichotomy on
	$k\ge b$ with constant projection $P(b)$  and constants $K$, $\alpha$.
	\medskip

	Let $\tilde\Phi(k,m)$ be the transition matrix of \eqref{eqext}.
	Of course $\tilde\Phi(k,m)=\Phi(k,m)$, the transition matrix for \eqref{eq}, when $a\le m\le k\le b$. We set
	\[
	\tilde P(k) = \left \{\begin{array}{ll}
	P(a) & \hbox{if $k\le a$}	\\
	P(k) & \hbox{if $a\le k\le b$}	\\
	P(b) & \hbox{if $k\ge b$}.
	\end{array}\right .\]
	
	We have dichotomies on $k\le a$, $a\le k\le b$ and $k\ge b$ with constants
	$K$ and $\alpha$ and projection $\tilde P(k)$. The invariance
	$\tilde A(k)\tilde P(k)=\tilde P(k+1)\tilde A(k)$ holds for
	$k\le a-1$, $a\le k\le b-1$, $b\le k$ and hence for all $k$.
	\medskip
	
	Also $\tilde\Phi(k,m)$ maps ${\cal N}\tilde P(m)$ bijectively
	onto ${\cal N}\tilde P(k)$ if $k\ge m$ since it holds for $a\le m\le k\le b$
	and also for $m\le k\le a$ and $b\le m\le k$ since ${\cal A}$ and ${\cal B}$
	are invertible. If, for example, $m<a$ and $k>b$, then
	\[ \tilde\Phi(k,m)={\cal B}^{k-b}\Phi(b,a){\cal A}^{a-m},\]
	which maps ${\cal N}\tilde P(m)$ bijectively
	onto ${\cal N}\tilde P(k)$ since ${\cal A}^{a-m}$ maps ${\cal N}\tilde P(m)$ bijectively onto ${\cal N}\tilde P(a)$, $\Phi(b,a)$ maps ${\cal N}\tilde P(a)$ bijectively onto ${\cal N}\tilde P(b)$ and ${\cal B}^{k-b}$ maps ${\cal N}\tilde P(b)$ bijectively onto ${\cal N}\tilde P(k)$. The other cases can be similarly treated.
	\medskip
	
	We have dichotomies on $k\le a$, $a\le k\le b$ and $k\ge b$ with constants
	$K$ and $\alpha$ and projection $\tilde P(k)$. So
	\[ |\tilde \Phi(k,m)\tilde P(m)|\le Ke^{-\alpha(k-m)},\quad
	|\tilde \Phi(m,k)\tilde(I-P(k))|\le Ke^{-\alpha(k-m)}\]
	if $m\le k\le a$ or $a\le k\le b$ or $b\le m\le k$. Suppose, for example,
	that $m<a$ and $b<k$. Then
	\begin{align*} |\tilde \Phi(k,m)\tilde P(m)|
	&=|\tilde \Phi(k,b)\tilde\Phi(b,a)\tilde\Phi(a,m)\tilde P(m)|\\
	&=|\tilde \Phi(k,b)\tilde P(b)\tilde\Phi(b,a)\tilde P(a)\tilde\Phi(a,m)\tilde P(m)|
	\quad{\rm by\;invariance}\\
	&=|e^{-\alpha(k-b)}P(b)\Phi(b,a)P(a)e^{-\alpha(a-m)}P(a)|\\
	&=|e^{-\alpha(k-b+a-m)}\Phi(b,a)P(a)|\\
	&\le e^{-\alpha(k-b+a-m)}Ke^{-\alpha(b-a)}\\
	&=Ke^{-\alpha(k-m)}.
	\end{align*}
	Next
	\begin{align*} |\tilde \Phi(m,k)(I-\tilde P(k))|
	&=|\tilde \Phi(m,a)\tilde\Phi(a,b)\tilde\Phi(b,k)(I-\tilde P(k))|\\
	&=|\tilde \Phi(m,a)(I-\tilde P(a))\tilde\Phi(a,b)(I-\tilde P(b))\tilde\Phi(b,k)(I-\tilde P(k))|\\
	&\quad\quad{\rm by\;invariance}\\
	&=|e^{-\alpha(a-m)}(I-P(a))\Phi(a,b)(I-P(b))e^{-\alpha(k-b)}(I-P(b))|\\
	&=|e^{-\alpha(a-m+k-b)}\Phi(a,b)(I-P(b))|\\
	&\le e^{-\alpha(k-b+a-m)}Ke^{-\alpha(b-a)}\\
	&=Ke^{-\alpha(k-m)}.
	\end{align*}
	Other cases can be similarly considered.
	\medskip
	
	This concludes the proof when $a,b\in\Z$. If $b=\infty$ the proof is the same but we do not need ${\cal B}$ and we just have the two intervals $k\le a$ and $k\ge a$.
	When $a=-\infty$, we do not need ${\cal A}$ and we just have the two intervals $k\le b$ and $k\ge b$.
	
\end{proof}

Then to prove Theorem \ref{roughthm} for an interval $[a,\infty)$, we extend $A(k)$ to $\tilde A(k)$ as in Lemma \ref{extending} and we extend the perturbation to $\Z$ by defining $B(k)=0$ for $k<a$. Then this perturbed system has a dichotomy on $\Z$ with constants as given in the theorem. Then we restrict back to the interval $[a,\infty)$. The interval $(-\infty,b]$ can be similarly considered. Thus the proof of Theorem \ref{roughthm} is completed.
\end{proof}

\begin{remark} Using a similar proof, we can prove an analogous theorem for differential equations. In this theorem it is shown that the constants for the perturbed system approach those for the unperturbed system as the perturbation tends to zero.
This solves a long standing problem; the roughness theorems in, for example, Coppel \cite{C} do not prove this. We include it here in case it has not been noticed before. Of course, in the discrete case, it was already proved in Henry \cite{H}.

\begin{theorem} \label{roughthmDE} Let $A(t)$ be a piecewise continuous $n\times n$ matrix function defined on an interval $J=[a,\infty)$ or $(-\infty,b]$ or $(-\infty,\infty)$. Suppose
\[ \dot x=A(t)x \]
has an exponential dichotomy on $J$ with projection $P(t)$ and constants $K$, $\alpha$. Then,  if $|B(t)|\le \delta$, where
\[2K\alpha^{-1}\delta<1,\]
the perturbed system
\[\dot x=[A(t)+B(t)]x\]
also has an exponential dichotomy on $J$ with projection $Q(t)$ of the same rank
with constants $L$ and $\beta$, where
\[\beta=\alpha\sqrt{1-2K\alpha^{-1}\delta},\quad L=\frac{2K(1+K\delta(1-2K\alpha^{-1}\delta)^{-1}\alpha^{-1})}
{1+\sqrt{1-2K\alpha^{-1}\delta}},\]
such that
\[|Q(t)-P(t)|\le 2KL(\alpha+\beta)^{-1}\delta.\]
\end{theorem}	
\end{remark}
\medskip

\section{Finite time conditions for dichotomy}

We study difference equations
\begin{equation}\label{lineq1} x(k+1)=A(k)x(k) \end{equation}
If $I$ is an interval of finite length, then for an arbitrary $\alpha>0$ and for an arbitrary invariant sequence $P(k)$ of projections
such that $A(k):{\cal N}P(k)\to {\cal N}P(k+1)$ is invertible, we can ensure equation \eqref{lineq1} has an {\it exponential dichotomy} on the interval $I$ simply by choosing $K$ appropriately. However, in spite of this, it turns out that this concept does have some use provided the constants $K$ and $\alpha$ are restricted suitably, as seen in the invertible case in \cite{PK1}.
\medskip

We could extend what was done in \cite{PK1} for the invertible case to the noninvertible case. However this has already been done in \cite{DMS} in a more general situation. Another related paper is \cite{BA}. We content ourselves here with stating the main result. It can be proved by a fairly straightforward modification of the proof in \cite{PK1}. Note in this
theorem we use Definition \ref{EDdef} as our definition of dichotomy.
\medskip

\begin{theorem} \label{thm2ft} Let $A(k)$ be matrix function defined on $\Z$ such that
$$ |A(k)|\le M $$
and such that the difference equation \eqref{lineq1} has for some fixed positive integer $N$ an exponential dichotomy on intervals
$[a,a+N]$ with uniform constant $K$ and exponent $\alpha$ for $a$ in a relatively dense set of integers (that is, there is a positive integer $L$ such that every interval of length $L$ in $Z$ contains a point of the set). Let
$$ \alpha>\beta>0,\quad \bar K>4K^8.$$
Then there is a function $N_0(M,K,\bar K,L,\alpha,\beta)$ such that if
\[N\ge  N_0(M,K,\bar K, L,\alpha,\beta),\]
 \eqref{lineq1} has an exponential dichotomy on $\Z$ with exponent $\beta$ and constant $\bar K$.
\end{theorem}

\end{document}